\documentclass[english]{article}
\usepackage[T1]{fontenc}
\usepackage{geometry}
\usepackage{color}
\usepackage{array}
\usepackage{float}
\usepackage{multirow}
\usepackage{amsmath}
\usepackage{amsthm}
\usepackage{amssymb}
\usepackage{graphicx}

\makeatletter

\providecommand{\tabularnewline}{\\}
\floatstyle{ruled}
\newfloat{algorithm}{tbp}{loa}
\providecommand{\algorithmname}{Algorithm}
\floatname{algorithm}{\protect\algorithmname}

\theoremstyle{plain}
\newtheorem{thm}{\protect\theoremname}[section]
\theoremstyle{plain}
\newtheorem{assumption}[thm]{\protect\assumptionname}
\theoremstyle{definition}
\newtheorem{defn}[thm]{\protect\definitionname}
\theoremstyle{plain}
\newtheorem{lem}[thm]{\protect\lemmaname}
\theoremstyle{remark}
\newtheorem{rem}[thm]{\protect\remarkname}
\theoremstyle{plain}
\newtheorem{prop}[thm]{\protect\propositionname}

\makeatother

\usepackage{babel}
\providecommand{\assumptionname}{Assumption}
\providecommand{\definitionname}{Definition}
\providecommand{\lemmaname}{Lemma}
\providecommand{\propositionname}{Proposition}
\providecommand{\remarkname}{Remark}
\providecommand{\theoremname}{Theorem}

\begin{document}
\title{Corporative Stochastic Approximation with Random Constraint Sampling
for Semi-Infinite Programming}
\author{Bo Wei, William B. Haskell, and Sixiang Zhao}
\maketitle
\begin{abstract}
We developed a corporative stochastic approximation (CSA) type algorithm
for semi-infinite programming (SIP), where the cut generation problem
is solved inexactly. First, we provide general error bounds for inexact
CSA. Then, we propose two specific random constraint sampling schemes
to approximately solve the cut generation problem. When the objective
and constraint functions are generally convex, we show that our randomized
CSA algorithms achieve an $\mathcal{O}(1/\sqrt{N})$ rate of convergence
in expectation (in terms of optimality gap as well as SIP constraint
violation). When the objective and constraint functions are all strongly
convex, this rate can be improved to $\mathcal{O}(1/N)$.
\end{abstract}

\section{Introduction}

In this paper, we combine the corporative stochastic approximation
(CSA) method developed in \cite{lan2016algorithms} with inexact cut
generation for semi-infinite programming (SIP). In particular, we
focus on random sampling methods to approximately solve the SIP cut
generation problem. The SIP cut generation problem is usually non-linear
and non-convex, so it is difficult to solve it to global optimality
deterministically. Two specific random constraint sampling schemes
are proposed to overcome this difficulty, and the randomized CSA algorithms
demonstrate good performance to solve SIP with theoretically guaranteed
convergence rates. 

\subsection{Previous work}

We refer the reader to \cite{bonnans2013perturbation,goberna2013semi,hettich1993semi,lopez2007semi,shapiro2009semi}
for recent detailed overviews of SIP. The main computational difficulty
in SIP comes from the infinitely many constraints, and several practical
schemes have been proposed to remedy this difficulty \cite{goberna2017recent,goberna2002linear,lopez2007semi,reemtsen1998numerical}.
We offer the following very rough classification of SIP methods based
on \cite{hettich1993semi,lopez2007semi,reemtsen1998numerical}. 

\textit{Exchange methods}: In exchange methods, in each iteration
a set of new constraints is exchanged for the previous set (there
are many ways to do this). Cutting plane methods are a special case
where constraints are never dropped. The algorithm in \cite{gribik1979central}
is the prototype for several SIP cutting plane schemes, and it has
been improved in various ways \cite{betro2004accelerated,kortanek1993central,mehrotra2014cutting}.
In particular, a new exchange method is proposed in \cite{zhang2010new}
that only keeps those active constraints with positive Lagrange multipliers.
New constraints are selected using a certain computationally-cheap
criterion. In \cite{mehrotra2014cutting}, the earlier central cutting
plane algorithm from \cite{kortanek1993central} is extended to allow
for nonlinear convex cuts.

Randomized cutting plane algorithms have recently been developed for
SIP in \cite{Calafiore_Uncertain_2005,campi2008exact,esfahani2015performance}.
The idea is to input a probability distribution over the constraints,
randomly sample a modest number of constraints, and then solve the
resulting relaxed problem. Intuitively, as long as a sufficient number
of samples of the constraints is drawn, the resulting randomized solution
should violate only a small portion of the constraints and achieve
near optimality. 

\textit{Discretization methods}: In the discretization approach,
a sequence of relaxed problems with a finite number of constraints
is solved according to a predefined or adaptively controlled grid
generation scheme \cite{reemtsen1991discretization,still2001discretization}.
Discretization methods are generally computationally expensive. The
convergence rate of the error between the solution of the SIP problem
and the solution of the discretized program is investigated in \cite{still2001discretization}. 

\textit{Local reduction methods}: In the local reduction approach,
an SIP problem is reduced to a problem with a finite number of constraints
\cite{gramlich1995local}. The reduced problem involves constraints
which are defined only implicitly, and the resulting problem is solved
via the Newton method which has good local convergence properties.
However, local reduction methods require strong assumptions and are
often conceptual.

\textit{Dual methods}: A wide class of SIP algorithms is based on
directly solving the KKT conditions. In \cite{ito2000dual,liu2002adaptive,liu2004new},
the authors derive Wolfe's dual for an SIP and discuss numerical schemes
for this problem. The KKT conditions often have some degree of smoothness,
and so various Newton-type methods can be applied \cite{li2004smoothing,ni2006truncated,qi2009smoothing,qi2003semismooth}.
However, feasibility is not guaranteed under the all Newton-type methods.
A new smoothing Newton-type method is proposed to overcome this drawback
in \cite{ling2010new}.

\textit{Applications}: SIP is the basis of the approximate linear
programming (ALP) approach for dynamic programming. Randomly sampling
state-action pairs is shown to give a tractable relaxed linear programming
problem, as explored in \cite{bhat2012non,esfahani2017infinite,deFarias_Sampling_2004}.
In \cite{bhat2012non,deFarias_Sampling_2004}, the sampling distribution
is assumed to be the occupation measure corresponding to the optimal
policy. In \cite{lin2017revisiting}, an adaptive constraint sampling
approach called 'ALP-Secant' is developed which is based on solving
a sequence of saddle-point problems. It is shown that ALP-Secant returns
a near optimal ALP solution and a lower bound on the optimal cost
with high probability in a finite number of iterations.

Many risk-aware optimization models also depend on SIP (e.g. \cite{noyan2013optimization,noyan2018optimization}),
in particular, risk-constrained optimization (e.g. \cite{dentcheva2003optimization,dentcheva2004optimality,dentcheva2009optimization,dentcheva2015optimization,haskell2017primal,homem2009cutting,hu2012sample}).
In \cite{dentcheva2003optimization,dentcheva2004optimality,dentcheva2009optimization,haskell2013optimization},
a duality theory for stochastic dominance constrained optimization
is developed which shows the special role of utility functions as
Lagrange multipliers. Relaxations of multivariate stochastic dominance
have been proposed based on various parametrized families of utility
functions, see \cite{dentcheva2009optimization,haskell2013optimization,homem2009cutting,hu2012sample}.
Computational aspects of the increasing concave stochastic dominance
constrained optimization are discussed in \cite{haskell2017primal,homem2009cutting,hu2012sample}.

\subsection{Contributions}

We summarize our main contributions in this work as follows:
\begin{enumerate}
\item We give error bounds for inexact CSA (where the cut generation problem
is solved inexactly). These error bounds are general, and may form
the basis for the convergence analysis of many CSA-type algorithms.
\item We develop two specialized CSA algorithms where random sampling is
used to approximately solve the cut generation problem. The first
algorithm is based on using a fixed sampling distribution, in line
with \cite{Calafiore_Uncertain_2005,campi2008exact,esfahani2015performance}.
Intuitively, as long as a sufficiently large number of samples is
drawn, the resulting randomized solution should violate only a \textquotedbl small
portion\textquotedbl{} of the constraints. The second algorithm is
based on adaptively sampling the constraints based on information
from the current iterate. In particular, we compute the analytical
solution of a regularized cut generation problem for the current iterate,
and then use this distribution to do adaptive sampling.
\item We provide a stochastic convergence analysis for both our specialized
CSA algorithms based on our general error bounds. We show that as
the errors in cut generation decrease at appropriate rates, our specialized
CSA algorithms achieve the same convergence rate as in the error-free
case. When the objective and constraint functions are convex, both
algorithms achieve an $\mathcal{O}(1/\sqrt{N})$ rate of convergence
in expectation, in terms of optimality gap and constraint violation.
If the objective and constraint functions are strongly convex, this
rate can be improved to $\mathcal{O}(1/N)$. 
\end{enumerate}
This paper is organized as follows. We first provide preliminary material
in Section \ref{prelim}. The following Section \ref{generalresults}
describes a general inexact CSA algorithm, and then provides error
bounds (in terms of the error in solving each cut generation problem).
Next, in Section \ref{examples}, we give the formal details for our
two specialized CSA algorithms and report their convergence rates.
For clearer organization, the detailed proofs of all our results are
gathered together in Section \ref{proofs}. We then present some numerical
experiments for CSA with random sampling in Section \ref{sec:Numerical-Experiments}.
Finally, we conclude the paper in Section \ref{conclusion} with a
discussion of further issues and future research.

\paragraph*{Notation}

We make use of the following basic notation throughout the paper.
For $x\in\mathbb{R}$, the ceiling function $\lceil x\rceil$ returns
the smallest integer greater than or equal to $x\in\mathbb{R}$. The
Euclidean norm and inner product on $\mathbb{R}^{n}$ are $\|x\|:=(\sum_{i=1}^{n}x_{i}^{2})^{\frac{1}{2}}$
and $\langle x,y\rangle=\sum_{i=1}^{n}x_{i}y_{i}$, respectively.
The Euclidean ball with radius $r$ centered at $x_{c}$ is $B_{r}(x_{c}):=\left\{ x:\,\left\Vert x-x_{c}\right\Vert \leq r\right\} $.
For a function $f\text{ : }\mathbb{R}^{n}\rightarrow\mathbb{R}$,
we denote its subdifferential by $\partial f(x)$ and a subgradient
of $f$ at $x$ by $f'(x)\in\partial f(x)$, respectively.

We also make use of the following further notation. For any set $\Delta\subset\mathbb{R}^{d}$,
$\mathcal{P}(\Delta)$ is the space of probability distributions on
$\Delta$. The Kullback-Liebler divergence is 
\[
D\left(\phi,\varphi\right):=\mathbb{E}_{\widetilde{\delta}\sim\phi}\left[\log\left(\frac{\phi(\widetilde{\delta})}{\varphi(\widetilde{\delta})}\right)\right]=\int_{\Delta}\log(\frac{\phi(\delta)}{\varphi(\delta)})\phi(d\delta)
\]
for probability densities $\phi,\varphi\in\mathcal{P}(\Delta)$. For
any integer $M\geq1$, we denote the $M-$Cartesian product of $\Delta$
by $\Delta^{M}:=\times_{i=1}^{M}\Delta$. Finally, for any probability
distribution $Q$ over set $\Delta$, the product measure and the
associated expectation on $\Delta^{M}$ are denoted by $Q^{M}$ and
$\mathbb{E}_{Q^{M}}$, respectively.

\section{\label{prelim}Preliminaries}

We begin our discussion of SIP with the following problem ingredients:
\begin{description}
\item [{A1}] Convex, compact decision set $\mathcal{X}\subset\mathbb{R}^{n}$;
\item [{A2}] Convex objective function $f\text{ : }\mathcal{X}\rightarrow\mathbb{R}$,
which is Lipschitz continuous with constant $L_{f}$;
\item [{A3}] Compact constraint index set $\Delta\subset\mathbb{R}^{d}$;
\item [{A4}] Constraint function $g\text{ : }\mathcal{X}\times\Delta\rightarrow\mathbb{R}$,
such that for each $\delta\in\Delta$, $x\rightarrow g(x,\delta)$
is convex and Lipschitz continuous with constant $L_{g,\mathcal{X}}$;
\item [{A5}] For all $x\in\mathcal{X}$, $\delta\rightarrow g(x,\delta)$
is Lipschitz continuous with constant $L_{g,\Delta}$.
\end{description}
We write the constraints as a single function $G(x):=\max_{\delta\in\Delta}g(x,\delta)$.
The resulting semi-infinite programming problem is:

\begin{eqnarray}
\min_{x\in\mathcal{X}}\Big\{ f(x):G(x):=\max_{\delta\in\Delta}g(x,\delta)\leq0\Big\}.\label{optiprob}
\end{eqnarray}
Problem (\ref{optiprob}) is a convex optimization problem under Assumptions
\textbf{A1}, \textbf{A2}, and \textbf{A4}. Formally, we also assume
that Problem (\ref{optiprob}) is solvable.
\begin{assumption}
An optimal solution $x^{*}$ of Problem (\ref{optiprob}) exists.
\end{assumption}

To continue, we recall some fundamental concepts of convex analysis.
\begin{defn}
A function $f:\,\mathcal{X}\rightarrow\mathbb{R}$ is strongly convex
with parameter $\alpha>0$, if for any $f'(x)\in\partial f(x)$ we
have
\[
f(x)\geq f(z)+\langle f'(z),x-z\rangle+\frac{\alpha}{2}\|x-z\|^{2},\forall x,z\in\mathcal{X}.
\]
\end{defn}

The distance generating function and its associated prox-function
are defined as follows. 
\begin{defn}
(i) A function $\omega_{X}:\mathcal{X}\rightarrow\mathbb{R}$ is a
distance generating function with parameter $\alpha>0$, if $\omega_{\mathcal{X}}$
is continuously differentiable and strongly convex with parameter
$\alpha$. 

(ii) (Bregman's distance) The prox-function associated with $\omega_{\mathcal{X}}$
is $V(x,z):=\omega_{\mathcal{X}}(z)-\omega_{\mathcal{X}}(x)-\langle\nabla\omega_{\mathcal{X}}(x),z-x\rangle$.

(iii) The prox-mapping is $P_{x,\mathcal{X}}(y):=\arg\min_{z\in\mathcal{X}}\{\langle y,z\rangle+V(x,z)\}$.
\end{defn}

Without loss of generality, we may assume that $\alpha=1$ in part
(i) of the preceding definition since we can always re-scale $\omega_{\mathcal{X}}(x)$
to become $\overline{\omega}_{\mathcal{X}}(x)=\omega_{\mathcal{X}}(x)/\alpha$.
The distance generating function $\omega_{\mathcal{X}}$ gives a measure
of the diameter of $\mathcal{X}$, i.e. $D_{\mathcal{X}}:=\sqrt{\max_{x,z\in\mathcal{X}}V(x,z)}$.
Clearly, the diameter satisfies $D_{\mathcal{X}}<\infty$ as long
as $\mathcal{X}$ is bounded.

We assume that the prox-function $V(x,z)$ is chosen such that the
prox-mapping $P_{x,\mathcal{X}}:\mathbb{R}^{n}\rightarrow\mathbb{R}^{n}$
can be easily computed. The next result follows from the definition
of the prox-function.
\begin{lem}
\cite[Lemma 2.1]{nemirovski2009robust}\label{block} For every $u,x\in\mathcal{X}$
and $y\in\mathbb{R}^{n}$, we have 
\[
V(P_{x,X}(y),u)\leq V(x,u)+\langle y,u-x\rangle+\frac{1}{2}\|y\|^{2}.
\]
\end{lem}

\section{\label{generalresults}General Error Bounds for Inexact CSA}

In this section, we derive general error bounds for inexact CSA applied
to Problem (\ref{optiprob}). These error bounds form the basis of
our convergence analysis for the two specialized CSA algorithms that
we consider in the next section.

The (general) CSA algorithm works as follows. We let $\left\{ x_{k}\right\} _{k\geq1}$
denote the sequence of iterates of the algorithm, $\left\{ \gamma_{k}\right\} _{k\geq1}$
a sequence of step-sizes with all $\gamma_{k}>0$, and $\left\{ \eta_{k}\right\} _{k\geq1}$
a sequence of error tolerances for constraint violation with all $\eta_{k}>0$.
At each iteration $k\geq1$, we need to solve the cut generation problem
\begin{eqnarray}
\max_{\delta\in\Delta}g(x_{k},\delta)\label{cut}
\end{eqnarray}
to determine if $x_{k}$ is feasible or to identify any violated constraints.
After we obtain 
\[
\delta_{k}\approx\arg\max_{\delta\in\Delta}g(x_{k},\delta),
\]
CSA performs a projected subgradient step with step-size $\gamma_{k}$
along either $f'(x_{k})$ or $g'(x_{k},\delta_{k})$, depending on
whether the condition $g(x_{k},\delta_{k})\leq\eta_{k}$ is satisfied
(i.e. depending on whether the constraint violation is below our error
tolerance or not).

Let $N$ denote the total number of iterations of the algorithm. For
some $1\leq s\leq N$, we may partition the indices 
\[
I:=\{s,\ldots,N\}
\]
into two subsets: 
\[
\mathcal{B}:=\{s\leq k\leq N\mid g(x_{k},\delta_{k})\leq\eta_{k}\}\quad\mbox{and}\quad\mathcal{N}:=I\backslash\mathcal{B}.
\]
The set $\mathcal{B}$ counts those iterations within $I$ for which
the constraint violation of $x_{k}$ corresponding to $\delta_{k}\approx\arg\max_{\delta\in\Delta}g(x_{k},\delta)$
is less than our tolerance $\eta_{k}$. When the algorithm terminates,
it returns the weighted average
\[
\overline{x}_{N,s}:=\frac{\sum_{k\in\mathcal{B}}\gamma_{k}x_{k}}{\sum_{k\in\mathcal{B}}\gamma_{k}}
\]
of iterates over $\mathcal{B}$ (which only indexes those iterates
where we believe the constraint violation is small). The general inexact
CSA algorithm is summarized in Algorithm \ref{CSAinSIP}.

\begin{algorithm}
\caption{The inexact CSA algorithm for SIP}
\label{CSAinSIP} \textbf{\textcolor{black}{Input:}} Number of iterations
$N$, initial point $x_{1}\in\mathcal{X}$, error tolerances $\{\eta_{k}\}_{k\geq1}$,
step-sizes $\{\gamma_{k}\}_{k\geq1}$.

\textbf{\textcolor{black}{For}} $k=1,2,\ldots,N$ \textbf{\textcolor{black}{do}}

Select $\delta_{k}\in\Delta$ such that $\delta_{k}\approx\arg\max_{\delta\in\Delta}g(x_{k},\delta)$.

Set 
\begin{eqnarray*}
h_{k}=\left\{ \begin{array}{l}
f'(x_{k}),\quad\quad\mbox{if \ensuremath{g(x_{k},\delta_{k})\leq\eta_{k}}},\\
g'(x_{k},\delta_{k}),\quad\mbox{otherwise}.
\end{array}\right.
\end{eqnarray*}
\begin{eqnarray*}
x_{k+1}=P_{x_{k},\mathcal{X}}(\gamma_{k}h_{k}).
\end{eqnarray*}
\textbf{\textcolor{black}{end for}} 

\textbf{\textcolor{black}{Output:}} $\overline{x}_{N,s}$. 
\end{algorithm}
The cut generation problem $\max_{\delta\in\Delta}g(x,\delta)$ is
typically a non-convex optimization problem. Generally speaking, there
is no fast algorithm that can solve this problem deterministically.
In our case, the error in each iteration comes from inexact solution
of $\max_{\delta\in\Delta}g(x_{k},\delta)$. We denote the error in
cut generation as

\[
\varepsilon_{k}:=G(x_{k})-g(x_{k},\delta_{k}),\,\forall k\geq1.
\]
Note that the errors $\left\{ \varepsilon_{k}\right\} _{k\geq1}$
are always nonnegative since $G\left(x\right)\geq g\left(x,\,\delta\right)$
for all $\delta\in\Delta$ by definition. 

Below we give a specific selection of the parameters $\{\eta_{k}\}_{k\geq1}$,
$\{\gamma_{k}\}_{k\geq1}$, and $s$ to be used in Algorithm \ref{CSAinSIP}:

\begin{equation}
\eta_{k}=\frac{6(L_{f}+L_{g,\mathcal{X}})D_{\mathcal{X}}}{\sqrt{k}},\,\gamma_{k}=\frac{D_{\mathcal{X}}}{\sqrt{k}(L_{f}+L_{g,\mathcal{X}})},\,k=1,2,\ldots,N,\,s=\lceil\frac{N}{2}\rceil,\label{variableparameters}
\end{equation}
for all $N\geq1$. The following result shows that $\overline{x}_{N,s}$
is well-defined under this policy. 
\begin{lem}
\label{nonemptyB} Suppose $\left\{ x_{k}\right\} _{k\geq1}$ is generated
by Algorithm \ref{CSAinSIP} with policy (\ref{variableparameters}),
then the set $\mathcal{B}\neq\emptyset$, i.e., $\overline{x}_{N,s}$
is well-defined. 
\end{lem}

Now we will bound the optimality gap and constraint violation of $\bar{x}_{N,\,s}$
in terms of the errors $\left\{ \varepsilon_{k}\right\} _{k\geq1}$
from inexact cut generation. The result of Theorem \ref{varistepresult}
is online since policy (\ref{variableparameters}) does not depend
on knowing $N$ in advance, and thus we may stop or continue the algorithm
anytime. In particular, the weighted average $\overline{x}_{N,\,s}$
from Theorem \ref{varistepresult} gives decreasing weight to older
iterates $\{x_{k}\}_{k\geq1}$.
\begin{thm}
\label{varistepresult} Suppose $\left\{ x_{k}\right\} _{k\geq1}$
is generated by Algorithm \ref{CSAinSIP} with policy (\ref{variableparameters}),
then for any $N\geq1$ we have
\[
f(\overline{x}_{N,s})-f(x^{*})\leq\frac{6D_{\mathcal{X}}(L_{f}+L_{g,\mathcal{X}})}{\sqrt{N}},
\]
and

\[
G(\overline{x}_{N,s})\leq\frac{12D_{\mathcal{X}}(L_{f}+L_{g,\mathcal{X}})}{\sqrt{N}}+\frac{\sum_{k\in\mathcal{B}}\varepsilon_{k}/\sqrt{k}}{\sum_{k\in\mathcal{B}}1/\sqrt{k}}.
\]
\end{thm}

\begin{rem}
The bound on the optimality gap does not depend on the errors $\{\varepsilon_{k}\}_{k\geq1}$
in cut generation, since objective function evaluations are error
free (in contrast to inexact evaluation of the constraint function
$G\left(x\right)$). 
\end{rem}

We can improve the $O\left(1/\sqrt{N}\right)$ convergence rate when
the objective function $f(\cdot)$ and the constraint functions $\left\{ g(\cdot,\delta)\right\} _{\delta\in\Delta}$
are all strongly convex. To proceed, we introduce a new assumption
on the quadratic growth of the prox-function $V(\cdot,\cdot)$.
\begin{assumption}
\label{strongconvex_and_L} (i) The objective function $f$ is strongly
convex with parameter $\mu_{f}>0$, and the constraint functions $g(\cdot,\delta)$
are all strongly convex with parameter $\mu_{g}>0$ (uniformly in
all $\delta\in\Delta$). 

(ii) There exists $L>0$, such that $V(x,z)\leq\frac{L}{2}\|x-z\|^{2},\forall x,z\in\mathcal{X}$. 
\end{assumption}

The constants in Assumption \ref{strongconvex_and_L} appear in our
parameter selection policy for the strongly convex case. For all $k=1,2,\ldots,N$,
let $\gamma_{k}$ be the step-sizes used in our algorithms, and denote
\[
a_{k}=\left\{ \begin{array}{l}
\frac{\mu_{f}\gamma_{k}}{L},\quad\mbox{if \ensuremath{g}(\ensuremath{x_{k}},\ensuremath{\delta_{k}})\ensuremath{\leq\eta_{k}}},\\
\frac{\mu_{g}\gamma_{k}}{L},\quad\text{otherwise},
\end{array}\right.A_{k}=\left\{ \begin{array}{l}
1,\quad\quad\quad\quad\quad\quad k=1,\\
(1-a_{k})A_{k-1},\quad2\leq k\leq N,
\end{array}\right.\mbox{and}\quad\rho_{k}=\frac{\gamma_{k}}{A_{k}}.
\]
For the strongly convex case, the output of Algorithm \ref{CSAinSIP}
is modified to 
\[
\overline{x}_{N,s}=\frac{\sum_{k\in\mathcal{B}}\rho_{k}x_{k}}{\sum_{k\in\mathcal{B}}\rho_{k}}.
\]
Our new policy is given as follows: for $k=1,2,\ldots,N,$

\begin{equation}
\eta_{k}=\frac{8L}{N}\max\left\{ \mu_{f},\mu_{g}\right\} \max\left\{ \frac{L_{f}^{2}}{\mu_{f}^{2}},\frac{L_{g,\mathcal{X}}^{2}}{\mu_{g}^{2}}\right\} ,\,\gamma_{k}=\left\{ \begin{array}{l}
\frac{2L}{\mu_{f}(k+1)},\quad\mbox{if \ensuremath{g}(\ensuremath{x_{k}},\ensuremath{\delta_{k}})\ensuremath{\leq\eta_{k}}},\\
\frac{2L}{\mu_{g}(k+1)},\quad\text{otherwise},
\end{array}\right.\,\;s=1.\label{stepsizeinstrongconvex}
\end{equation}

The following result shows that $\overline{x}_{N,s}$ is well-defined
for this policy as well.
\begin{lem}
\label{nonemptyBstronglyconvex} Suppose Assumption \ref{strongconvex_and_L}
holds. Suppose $\left\{ x_{k}\right\} _{k\geq1}$ is generated by
Algorithm \ref{CSAinSIP} with policy (\ref{stepsizeinstrongconvex}),
then the set $\mathcal{B}\neq\emptyset$, i.e., $\overline{x}_{N,s}$
is well-defined.
\end{lem}

Now we give an improved error bound for inexact CSA under policy (\ref{stepsizeinstrongconvex})
for the strongly convex case.
\begin{thm}
\label{strongconvexresult} Suppose Assumption \ref{strongconvex_and_L}
holds. Let $\left\{ x_{k}\right\} _{k\geq1}$ be generated by Algorithm
\ref{CSAinSIP} with policy (\ref{stepsizeinstrongconvex}), then
for any $N\geq1$ we have
\[
f(\overline{x}_{N,s})-f(x^{*})\leq\frac{8L}{N+1}\max\left\{ \mu_{f},\mu_{g}\right\} \max\left\{ \frac{L_{f}^{2}}{\mu_{f}^{2}},\frac{L_{g,\mathcal{X}}^{2}}{\mu_{g}^{2}}\right\} ,
\]
and
\[
G(\overline{x}_{N,s})\leq\frac{8L}{N}\max\left\{ \mu_{f},\mu_{g}\right\} \max\left\{ \frac{L_{f}^{2}}{\mu_{f}^{2}},\frac{L_{g,\mathcal{X}}^{2}}{\mu_{g}^{2}}\right\} +\frac{\sum_{k\in\mathcal{B}}k\,\varepsilon_{k}}{\sum_{k\in\mathcal{B}}k}.
\]
\end{thm}

\begin{rem}
In the strongly convex case, the convergence rate may be improved
to $O\left(1/N\right)$ if the errors in cut generation decrease at
appropriate rate. 
\end{rem}

\section{\label{examples} Random Constraint Sampling}

As we have already pointed out, the cut generation Problem (\ref{cut})
is a general nonlinear non-convex optimization problem, and there
is no fast algorithm that can solve such a problem deterministically.
In this section, we describe two random constraint sampling schemes
that can approximately solve the cut generation problem. The first
scheme is based on sampling from a fixed probability distribution
(Subsection \ref{fixed sampling implementation}), while the second
scheme is based on sampling adaptively from a probability distribution
that is updated in each iteration based on the current iterate (Subsection
\ref{adaptive sampling implementation}).

\subsection{\label{fixed sampling implementation}Fixed Constraint Sampling }

In this subsection, we approximately solve the cut generation Problem
(\ref{cut}) by sampling from a fixed distribution on $\Delta$. To
begin, we take a probability distribution $Q$ on $\Delta$ as user
input. To solve Problem (\ref{cut}) at iteration $k\geq1$, we let
$\delta_{k}^{(1)},\delta_{k}^{(2)},\ldots,\delta_{k}^{(M_{k})}$ (where
$M_{k}\geq1$ is the sample size for all $k\geq1$) be independent
identically distributed (i.i.d.) samples from $\Delta$ generated
according to $Q$. Then, we define 
\[
\delta_{k}\in\arg\max_{i=1,\ldots,\,M_{k}}g\left(x_{k},\,\delta_{k}^{\left(i\right)}\right)
\]
to be the element among $\left\{ \delta_{k}^{(i)}\right\} _{i=1}^{M_{k}}$
which maximizes $\left\{ g\left(x_{k},\,\delta_{k}^{\left(i\right)}\right)\right\} _{i=1}^{M_{k}}$.

We need the following assumption on the sampling distribution $Q$. 
\begin{assumption}
\label{strictlyincr} There exists a strictly increasing function
$\varphi:\mathbb{R}_{+}\rightarrow[0,1]$ such that $Q\{B_{r}(\delta)\}\geq\varphi(r)$,
for all $\delta\in\Delta$ and all open balls $B_{r}(\delta)\subset\Delta$.
\end{assumption}

The above assumption means that $Q$ has support on all of $\Delta$,
it also appears in Proposition 3.8 of \cite{esfahani2015performance}.
For more discussion, the reader is referred to Assumption 3.1 of \cite{kanamori2012worst}.

Intuitively, as long as the number of samples $M$ is large enough,
we expect $\max_{1\leq i\leq M}g(x,\delta^{(i)})$ will be close to
$G(x)$ with high probability with respect to $Q^{M}$. We have a
result in expectation for the approximation quality. For $\varepsilon$,
$\beta$ in $(0,1)$, we define
\[
M(\varepsilon,\beta):=\lceil\frac{\ln\beta}{\ln(1-\varepsilon)}\rceil,
\]
which will appear in the next result to denote the threshold of sample
size. Denote the lower bound and upper bound of $g(x,\delta)$ over
$x\in\mathcal{X},\delta\in\Delta$ as $\underline{M}$ and $\overline{M}$
(due to the continuity of $(x,\delta)\rightarrow g(x,\delta)$, and
the compactness of $\mathcal{X}$ and $\Delta$), respectively, i.e.,
\[
\underline{M}\leq\min_{x\in\mathcal{X},\delta\in\Delta}g(x,\delta)\leq\max_{x\in\mathcal{X},\delta\in\Delta}g(x,\delta)\leq\overline{M}.
\]

\begin{prop}
\label{keyrandcut-inexpectation} Suppose Assumption \ref{strictlyincr}
holds. Given $\epsilon>0$, for $M\geq M(\varphi(\frac{\epsilon}{2L_{g,\Delta}}),\frac{\epsilon}{2(\overline{M}-\underline{M})})$
i.i.d. samples generated from $Q$, we have $\mathbb{E}_{Q^{M}}\left[\max_{1\leq i\leq M}g(x,\delta^{(i)})\right]\geq G(x)-\epsilon$.
\end{prop}

We now investigate the convergence of inexact CSA based on this fixed
sampling scheme. We define 
\[
\mathcal{Q}=\mathcal{Q}(\left\{ M_{k}\right\} _{k\in\mathcal{B}}):=\times_{k\in\mathcal{B}}Q^{M_{k}}
\]
to be the probability distribution of the samples $\Big\{\{\delta_{k}^{(i)}\}_{i=1}^{M_{k}}\Big\}_{k\in\mathcal{B}}$
on the space $\times_{k\in\mathcal{B}}\Delta^{M_{k}}$.
\begin{thm}
\label{randvaristep} Suppose Assumption \ref{strictlyincr} holds.
Suppose $\left\{ x_{k}\right\} _{k\geq1}$ is generated by Algorithm
\ref{CSAinSIP} under policy (\ref{variableparameters}). Take $\epsilon_{k}=(L_{f}+L_{g,\mathcal{X}})D_{\mathcal{X}}/\sqrt{k}$,
and $M_{k}\geq M(\varphi(\frac{\epsilon_{k}}{2L_{g,\Delta}}),\frac{\epsilon_{k}}{2(\overline{M}-\underline{M})})$
for all $k\geq1$. Then, for any $N\geq1$, we have
\begin{eqnarray*}
f(\overline{x}_{N,s})-f(x^{*})\leq\frac{6D_{\mathcal{X}}(L_{f}+L_{g,\mathcal{X}})}{\sqrt{N}},
\end{eqnarray*}
and 
\begin{eqnarray*}
\mathbb{E}_{\mathcal{Q}}\left[G(\overline{x}_{N,s})\right]\leq\frac{14D_{\mathcal{X}}(L_{f}+L_{g,\mathcal{X}})}{\sqrt{N}}.
\end{eqnarray*}
\end{thm}

In view of Theorem \ref{randvaristep} we see that inexact CSA with
fixed random constraint sampling achieves an $\mathcal{O}(1/\sqrt{N})$
rate of convergence in expectation (with respect to $\mathcal{Q}$)
for solving Problem (\ref{optiprob}) in the general convex case.
Next we consider an improved convergence rate for the strongly convex
case.
\begin{thm}
\label{randstrong} Suppose Assumptions \ref{strongconvex_and_L}
and \ref{strictlyincr} hold. Suppose $\left\{ x_{k}\right\} _{k\geq1}$
is generated by Algorithm \ref{CSAinSIP} under policy (\ref{stepsizeinstrongconvex}).
Take $\epsilon_{k}=\frac{L}{N}\max\left\{ \mu_{f},\mu_{g}\right\} \max\left\{ \frac{L_{f}^{2}}{\mu_{f}^{2}},\frac{L_{g,\mathcal{X}}^{2}}{\mu_{g}^{2}}\right\} $,
and $M_{k}\geq M(\varphi(\frac{\epsilon_{k}}{2L_{g,\Delta}}),\frac{\epsilon_{k}}{2(\overline{M}-\underline{M})})$
for all $k\geq1$. Then, for any $N\geq1$, we have
\begin{eqnarray*}
f(\overline{x}_{N,s})-f(x^{*})\leq\frac{8L}{N+1}\max\left\{ \mu_{f},\mu_{g}\right\} \max\left\{ \frac{L_{f}^{2}}{\mu_{f}^{2}},\frac{L_{g,\mathcal{X}}^{2}}{\mu_{g}^{2}}\right\} ,
\end{eqnarray*}
and
\begin{eqnarray*}
\mathbb{E}_{\mathcal{Q}}\left[G(\overline{x}_{N,s})\right]\leq\frac{9L}{N}\max\left\{ \mu_{f},\mu_{g}\right\} \max\left\{ \frac{L_{f}^{2}}{\mu_{f}^{2}},\frac{L_{g,\mathcal{X}}^{2}}{\mu_{g}^{2}}\right\} .
\end{eqnarray*}
\end{thm}

\subsection{\label{adaptive sampling implementation}Adaptive Constraint Sampling }

In this subsection, we consider an alternative \textit{adaptive}
constraint sampling scheme for the cut generation Problem (\ref{cut}).
In particular, in iteration $k$ we will construct a constraint sampling
distribution that is tailored to the current iterate $x_{k}$. More
specifically, for any $\epsilon>0$ and $x\in\mathcal{X}$, we want
to find a probability distribution $\phi=\phi\left(x,\,\epsilon\right)\in\mathcal{P}(\Delta)$
(which depends on $x$ and $\epsilon$) on $\Delta$, such that 
\[
\mathbb{E}_{\widetilde{\delta}\sim\phi}\left[g\left(x,\widetilde{\delta}\right)\right]\geq G(x)-\epsilon,
\]
which guarantees that the samples generated from this distribution
are very likely to solve our cut generation Problem (\ref{cut}).
Then, in each iteration we will construct such a distribution from
$x_{k}$, and use it to guide our next round of random constraint
sampling.

To continue, we introduce a new assumption on the set $\Delta$.
\begin{assumption}
\label{fulldimandconvex} The set $\Delta\subset\mathbb{R}^{d}$ is
full dimensional and convex.
\end{assumption}

The following preliminary lemma is key for our adaptive sampling scheme.
It establishes an equivalence between the general nonlinear finite-dimensional
optimization problem $\max_{\delta\in\Delta}g(x,\delta)$ and the
infinite-dimensional \textit{linear} optimization problem 
\[
\max_{\phi\in\mathcal{P}(\Delta)}\mathbb{E}_{\widetilde{\delta}\sim\phi}\left[g\left(x,\widetilde{\delta}\right)\right]
\]
in probability distributions. 
\begin{lem}
\label{transform}For all $x\in\mathcal{X}$, $G(x)=\max_{\delta\in\Delta}g(x,\delta)=\max_{\phi\in\mathcal{P}(\Delta)}\mathbb{E}_{\widetilde{\delta}\sim\phi}\left[g\left(x,\widetilde{\delta}\right)\right]$. 
\end{lem}

Let $\phi_{u}$ denote the uniform probability distribution on $\Delta$,
that is, 
\[
\phi_{u}(\delta)=p_{u}:=\left(\int_{\Delta}d\delta\right)^{-1},\,\forall\delta\in\Delta.
\]
We define a regularized cut generation problem as follows,
\begin{equation}
\max_{\phi\in\mathcal{P}(\Delta)}\mathbb{E}_{\widetilde{\delta}\sim\phi}\left[g\left(x,\,\widetilde{\delta}\right)\right]-\kappa\,D\left(\phi,\,\phi_{u}\right),\label{regulcutgen}
\end{equation}
where $\kappa\in(0,1]$ is the regularization parameter. The mapping
$\phi\rightarrow D\left(\phi,\phi_{u}\right)$ is convex, thus the
regularized cut generation Problem (\ref{regulcutgen}) is an infinite-dimensional
convex optimization problem. We can expect that if the regularization
parameter $\kappa$ is small enough, the solution of the regularized
Problem (\ref{regulcutgen}) provides useful information to solve
our cut generation Problem (\ref{cut}).

We will show that the regularized cut generation Problem (\ref{regulcutgen})
is well defined. In particular, we show that the maximizer (which
depends on $x$ and $\kappa$)
\[
\phi_{\kappa,\,x}\in\arg\max_{\phi\in\mathcal{\mathcal{P}}(\Delta)}\left\{ \mathbb{E}_{\widetilde{\delta}\sim\phi}\left[g\left(x,\,\widetilde{\delta}\right)\right]-\kappa\,D\left(\phi,\,\phi_{u}\right)\right\} ,
\]
is attained and is given in closed form. The next lemma is based on
calculus of variations.
\begin{lem}
\label{existenceofmaximizer} For any $\kappa\in(0,1]$ and $x\in\mathcal{X}$,
the maximizer of Problem (\ref{regulcutgen}) is attained, and it
is
\[
\phi_{\kappa,\,x}(\delta)=\frac{\exp\left(g(x,\delta)/\kappa\right)}{\int_{\Delta}\exp\left(g(x,\delta)/\kappa\right)d\delta},\,\forall\delta\in\Delta.
\]
\end{lem}

Since $\Delta\subset\mathbb{R}^{d}$ is full dimensional, we may let
$R_{\Delta}$ be the radius of the largest ball which can be included
in $\Delta$. Specifically, there exists $\delta_{0}\in\Delta$ such
that the Euclidean ball $B_{R_{\Delta}}(\delta_{0})\subseteq\Delta$.
Define 
\[
r:=\left(\int_{B_{R_{\Delta}}(\delta_{0})}1d\delta\right)/\left(\int_{\Delta}d\delta\right)
\]
to be the ratio between the volume of the largest such ball $B_{R_{\Delta}}(\delta_{0})$
and the volume of $\Delta$ (necessarily $r\leq1$). The following
result demonstrates that the gap between the cut generation Problem
(\ref{cut}) and its regularization (\ref{regulcutgen}) can be made
arbitrarily small through our control of $\epsilon$. Let $D_{\Delta}:=\max_{\delta,\delta'\in\Delta}\left\Vert \delta-\delta'\right\Vert $
denote the Euclidean diameter of $\Delta$ and define $C:=L_{g,\Delta}(R_{\Delta}+D_{\Delta})-\log(r)$.
We also define

\[
\kappa(\epsilon):=\min\left\{ \frac{\epsilon}{2C},\left(\frac{\epsilon}{2d}\right)^{2},1\right\} ,\,\forall\epsilon>0.
\]

\begin{prop}
\label{maxsupdiff} Suppose Assumption \ref{fulldimandconvex} holds
and choose $\epsilon>0$. For any $x\in\mathcal{X}$,

\[
\max_{\phi\in\mathcal{P}(\Delta)}\left\{ \mathbb{E}_{\widetilde{\delta}\sim\phi}\left[g\left(x,\widetilde{\delta}\right)\right]-\kappa(\epsilon)\,D\left(\phi,\phi_{u}\right)\right\} \geq G(x)-\epsilon,
\]
and 

\begin{equation}
G(x)\geq\mathbb{E}_{\widetilde{\delta}\sim\phi_{\mathfrak{\kappa(\epsilon)},\,x}}\left[g\left(x,\widetilde{\delta}\right)\right]\geq G(x)-\epsilon.\label{distributionalcutgen}
\end{equation}
\end{prop}

From (\ref{distributionalcutgen}), we see that the solution of the
regularized cut generation Problem (\ref{regulcutgen}) provides a
solution of the inequality $\mathbb{E}_{\widetilde{\delta}\sim\phi_{\mathfrak{\kappa(\epsilon)},\,x}}\left[g\left(x,\widetilde{\delta}\right)\right]\geq G(x)-\epsilon$.

The adaptive constraint sampling scheme works as follows. Suppose
we are given tolerances $\left\{ \epsilon_{k}\right\} _{k\geq1}$
with $\epsilon_{k}>0$ for all $k\geq1$. At iteration $k\geq1$,
we sample from the probability density $\phi_{\mathfrak{\kappa}(\epsilon_{k}),\,x_{k}}$.
Let $\delta_{k}^{(1)},\delta_{k}^{(2)},\ldots,\delta_{k}^{(M_{k})}$
(with $M_{k}\geq1$) be i.i.d. samples from $\Delta$ generated according
to $\phi_{\mathfrak{\kappa}(\epsilon_{k}),\,x_{k}}$, and again define
$\delta_{k}\in\arg\max_{i=1,\ldots,\,M_{k}}g\left(x_{k},\,\delta_{k}^{\left(i\right)}\right)$
to be a maximizer of $\{g(x_{k},\delta_{k}^{(i)})\}_{i=1}^{M_{k}}$. 
\begin{prop}
\label{cutgenerationinexpectation_adaptive} Suppose Assumption \ref{fulldimandconvex}
holds. Given $\epsilon>0$ and $x\in\mathcal{X}$, for any $M\geq1$,
let $\delta^{\left(1\right)},\ldots,\,\delta^{\left(M\right)}$ be
i.i.d. samples from $\phi_{\mathfrak{\kappa}(\epsilon),\,x}$, then
$\mathbb{E}_{\phi_{\kappa(\epsilon),\,x}^{M}}\left[\max_{1\leq i\leq M}g(x,\delta^{(i)})\right]\geq G(x)-\epsilon$.
\end{prop}

Now we consider the convergence rate of the adaptive sampling scheme.
We define the distribution
\[
\mathcal{P}=\mathcal{P}(\left\{ x_{k}\right\} _{k\in\mathcal{B}},\left\{ M_{k}\right\} _{k\in\mathcal{B}}):=\times_{k\in\mathcal{B}}\phi_{\mathfrak{\kappa}(\epsilon_{k}),\,x_{k}}^{M_{k}}
\]
of the samples $\Big\{\{\delta_{k}^{(i)}\}_{i=1}^{M_{k}}\Big\}_{k\in\mathcal{B}}$
on the space $\times_{k\in\mathcal{B}}\Delta^{M_{k}}$. 
\begin{thm}
\label{varistep-algo2} Suppose Assumption \ref{fulldimandconvex}
holds. Suppose $\left\{ x_{k}\right\} _{k\geq1}$ is generated according
to Algorithm \ref{CSAinSIP} with policy (\ref{variableparameters}).
For each $k\geq1$ and $\epsilon_{k}=(L_{f}+L_{g,\mathcal{X}})D_{\mathcal{X}}/\sqrt{k}$,
we generate $M_{k}\geq1$ i.i.d. samples according to $\phi_{\mathfrak{\kappa}(\epsilon_{k}),\,x_{k}}$.
Then, for any $N\geq1$,
\[
f(\overline{x}_{N,s})-f(x^{*})\leq\frac{6D_{\mathcal{X}}(L_{f}+L_{g,\mathcal{X}})}{\sqrt{N}},
\]
and 

\[
\mathbb{E}_{\mathcal{P}}\left[G(\overline{x}_{N,s})\right]\leq\frac{14(L_{f}+L_{g,\mathcal{X}})D_{\mathcal{X}}}{\sqrt{N}}.
\]
\end{thm}

As for the fixed sampling scheme, we find an improved convergence
rate for the strongly convex case under the adaptive sampling scheme
as well.
\begin{thm}
\label{strong-algo2} Suppose Assumptions \ref{strongconvex_and_L}
and \ref{fulldimandconvex} hold. Suppose $\left\{ x_{k}\right\} _{k\geq1}$
is generated according to Algorithm \ref{CSAinSIP} with policy (\ref{stepsizeinstrongconvex}).\textup{
}For each $k\geq1$ and $\epsilon_{k}=\frac{L}{N}\max\left\{ \mu_{f},\mu_{g}\right\} \max\left\{ \frac{L_{f}^{2}}{\mu_{f}^{2}},\frac{L_{g,\mathcal{X}}^{2}}{\mu_{g}^{2}}\right\} $,
we generate $M_{k}\geq1$ i.i.d. samples according to $\phi_{\mathfrak{\kappa}(\epsilon_{k}),\,x_{k}}$.
Then, for any $N\geq1$,

\[
f(\overline{x}_{N,s})-f(x^{*})\leq\frac{8L}{N+1}\max\left\{ \mu_{f},\mu_{g}\right\} \max\left\{ \frac{L_{f}^{2}}{\mu_{f}^{2}},\frac{L_{g,\mathcal{X}}^{2}}{\mu_{g}^{2}}\right\} ,
\]
and 

\[
\mathbb{E}_{\mathcal{P}}\left[G(\overline{x}_{N,s})\right]\leq\frac{9L}{N}\max\left\{ \mu_{f},\mu_{g}\right\} \max\left\{ \frac{L_{f}^{2}}{\mu_{f}^{2}},\frac{L_{g,\mathcal{X}}^{2}}{\mu_{g}^{2}}\right\} .
\]
\end{thm}

In view of Theorem \ref{strong-algo2}, inexact CSA with adaptive
sampling achieves an $\mathcal{O}(1/N)$ rate of convergence in expectation,
in terms of the optimality gap and constraint violation, in the strongly
convex case. 
\begin{rem}
Through Proposition \ref{keyrandcut-inexpectation} and Proposition
\ref{cutgenerationinexpectation_adaptive}, we see two major differences
between the fixed sampling and adaptive sampling schemes. First, the
fixed sampling scheme requires batch samples, while only one sample
per iteration is needed to make the adaptive sampling scheme work
due to the inequality $\mathbb{E}_{\widetilde{\delta}\sim\phi_{\mathfrak{\kappa(\epsilon)},\,x}}\left[g\left(x,\widetilde{\delta}\right)\right]\geq G(x)-\epsilon$.
Of course, we get better performance if we use batch sampling under
the adaptive sampling since we always have $\max_{1\leq i\leq M}g(x,\delta^{(i)})\geq g(x,\delta^{(i)})$
for each $1\leq i\leq M$. Second, $\mathcal{P}=\times_{k\in\mathcal{B}}\phi_{\mathfrak{\kappa}(\epsilon_{k}),\,x_{k}}^{M_{k}}$
depends on the error tolerances and the current iterates under the
adaptive sampling, while $\mathcal{Q}=\times_{k\in\mathcal{B}}Q^{M_{k}}$
does not under the fixed sampling scheme. There is a trade-off between
the two sampling schemes. Under the fixed sampling scheme, we do not
need to change the sampling distribution iteration by iteration, but
it requires batch samples to achieve a desired cut generation tolerance.
Under the adaptive sampling scheme, we need to generate different
sampling distributions at different iterations, but the required number
of samples is much smaller. 
\end{rem}

\section{\label{proofs} Proofs of Main Results}

In this section, we provide the proofs for our main results. In Subsection
\ref{subsec:General-Error-Bounds}, we establish the general error
bounds for inexact CSA (Theorem \ref{varistepresult} for the generally
convex case and Theorem \ref{strongconvexresult} for the strongly
convex case). The details of the fixed sampling cut generation scheme
(Proposition \ref{keyrandcut-inexpectation}) and the corresponding
CSA convergence results (Theorems \ref{randvaristep} and \ref{randstrong})
are in Subsection \ref{subsec:Fixed-Sampling}. All material for the
adaptive sampling cut generation scheme (Proposition \ref{cutgenerationinexpectation_adaptive})
and the corresponding CSA convergence results (Theorems \ref{varistep-algo2}
and \ref{strong-algo2}) are in Subsection \ref{subsec:Adaptive-Sampling}.

\subsection{\label{subsec:General-Error-Bounds} General Error Bounds Analysis
for Inexact CSA}

\subsubsection{\label{subsec:General-Convex-Case} General Convex Case}

The following preliminary result establishes an important recursion
for CSA.
\begin{prop}
For stepsizes $\{\gamma_{k}\}_{k\geq1}$, tolerances $\{\eta_{k}\}_{k\geq1}$,
and $1\leq s\leq N$ in Algorithm \ref{CSAinSIP}, we have 

\begin{equation}
\sum_{k\in\mathcal{N}}\gamma_{k}(\eta_{k}-g(x,\delta_{k}))+\sum_{k\in\mathcal{B}}\gamma_{k}\langle f'(x_{k}),x_{k}-x\rangle\leq V(x_{s},x)+\frac{1}{2}\sum_{k\in\mathcal{B}}\gamma_{k}^{2}\|f'(x_{k})\|^{2}+\frac{1}{2}\sum_{k\in\mathcal{N}}\gamma_{k}^{2}\|g'(x_{k},\delta_{k})\|^{2},\label{recursion}
\end{equation}
for all $x\in\mathcal{X}$. 
\end{prop}

\begin{proof}
For any $s\leq k\leq N$, using Lemma \ref{block}, we have 
\begin{eqnarray}
V(x_{k+1},x)\leq V(x_{k},x)+\gamma_{k}\langle h_{k},x-x_{k}\rangle+\frac{1}{2}\gamma_{k}^{2}\|h_{k}\|^{2}.\label{recur}
\end{eqnarray}
Observe that if $k\in\mathcal{B}$, then $h_{k}=f'(x_{k})$ and $\langle h_{k},x_{k}-x\rangle=\langle f'(x_{k}),x_{k}-x\rangle$.
Moreover, if $k\in\mathcal{N}$, then $h_{k}=g'(x_{k},\delta_{k})$
and 
\[
\langle h_{k},x_{k}-x\rangle=\langle g'(x_{k},\delta_{k}),x_{k}-x\rangle\geq g(x_{k},\delta_{k})-g(x,\delta_{k})\geq\eta_{k}-g(x,\delta_{k}).
\]
Summing up the inequalities in (\ref{recur}) from $k=s$ to $N$
and using the previous two observations, we obtain 

\[
\begin{array}{rcl}
V(x_{N+1},x) & \leq & V(x_{s},x)-\sum_{k=s}^{N}\gamma_{k}\langle h_{k},x_{k}-x\rangle+\frac{1}{2}\sum_{k=s}^{N}\gamma_{k}^{2}\|h_{k}\|^{2}\\
 & = & V(x_{s},x)-\big[\sum_{k\in\mathcal{N}}\gamma_{k}\langle g'(x_{k},\delta_{k}),x_{k}-x\rangle+\sum_{k\in\mathcal{B}}\gamma_{k}\langle f'(x_{k}),x_{k}-x\rangle\big]+\frac{1}{2}\sum_{k=s}^{N}\gamma_{k}^{2}\|h_{k}\|^{2}\\
 & \leq & V(x_{s},x)-\big[\sum_{k\in\mathcal{N}}\gamma_{k}(\eta_{k}-g(x,\delta_{k}))+\sum_{k\in\mathcal{B}}\gamma_{k}\langle f'(x_{k}),x_{k}-x\rangle\big]\\
 &  & +\frac{1}{2}\sum_{k\in\mathcal{B}}\gamma_{k}^{2}\|f'(x_{k})\|^{2}+\frac{1}{2}\sum_{k\in\mathcal{N}}\gamma_{k}^{2}\|g'(x_{k},\delta_{k})\|^{2}.
\end{array}
\]
\end{proof}
We next present a sufficient condition for the output $\overline{x}_{N,s}$
to be well-defined. 
\begin{lem}
\label{lemforsuff} Suppose 
\begin{eqnarray}
\frac{N-s+1}{2}\min_{k\in\mathcal{N}}\gamma_{k}\eta_{k}>D_{\mathcal{X}}^{2}+\frac{1}{2}\sum_{k\in\mathcal{B}}\gamma_{k}^{2}L_{f}^{2}+\frac{1}{2}\sum_{k\in\mathcal{N}}\gamma_{k}^{2}L_{g,\mathcal{X}}^{2}\label{suff}
\end{eqnarray}
holds. Then $\mathcal{B}\neq\emptyset$, i.e., $\overline{x}_{N,s}$
is well-defined. Furthermore, either (i) $|\mathcal{B}|\geq(N-s+1)/2$
or (ii) $\sum_{k\in\mathcal{B}}\gamma_{k}\langle f'(x_{k}),x_{k}-x^{*}\rangle<0$. 
\end{lem}

\begin{proof}
Fixing $x=x^{*}$ in (\ref{recursion}) gives

\[
\sum_{k\in\mathcal{N}}\gamma_{k}(\eta_{k}-g(x^{*},\delta_{k}))+\sum_{k\in\mathcal{B}}\gamma_{k}\langle f'(x_{k}),x_{k}-x^{*}\rangle\leq V(x_{s},x^{*})+\frac{1}{2}\sum_{k\in\mathcal{B}}\gamma_{k}^{2}L_{f}^{2}+\frac{1}{2}\sum_{k\in\mathcal{N}}\gamma_{k}^{2}L_{g,\mathcal{X}}^{2}.
\]
If $\sum_{k\in\mathcal{B}}\gamma_{k}\langle f'(x_{k}),x_{k}-x^{*}\rangle\geq0$,
then since $g(x^{*},\delta_{k})\leq G(x^{*})\leq0$ we have 
\begin{eqnarray}
\sum_{k\in\mathcal{N}}\gamma_{k}\eta_{k}\leq V(x_{s},x^{*})+\frac{1}{2}\sum_{k\in\mathcal{B}}\gamma_{k}^{2}L_{f}^{2}+\frac{1}{2}\sum_{k\in\mathcal{N}}\gamma_{k}^{2}L_{g,\mathcal{X}}^{2}.\label{contr}
\end{eqnarray}
Suppose that $|\mathcal{B}|<(N-s+1)/2$, i.e., $|\mathcal{N}|\geq(N-s+1)/2$.
Then, 
\[
\sum_{k\in\mathcal{N}}\gamma_{k}\eta_{k}\geq\frac{N-s+1}{2}\min_{k\in\mathcal{N}}\gamma_{k}\eta_{k}>D_{\mathcal{X}}^{2}+\frac{1}{2}\sum_{k\in\mathcal{B}}\gamma_{k}^{2}L_{f}^{2}+\frac{1}{2}\sum_{k\in\mathcal{N}}\gamma_{k}^{2}L_{g,\mathcal{X}}^{2},
\]
which contradicts (\ref{contr}). Thus, condition (i) holds. Alternatively,
if $\sum_{k\in\mathcal{B}}\gamma_{k}\langle f'(x_{k}),x_{k}-x^{*}\rangle<0$,
then condition (ii) holds. 
\end{proof}
We can now prove Lemma \ref{nonemptyB} based on the above result.
\begin{proof}[Proof of Lemma \ref{nonemptyB}]
 For $\{\gamma_{k}\}_{k\geq1}$, $\{\eta_{k}\}_{k\geq1}$, and $s$
chosen as in (\ref{variableparameters}), for $\mathcal{N}\subset\{\lceil N/2\rceil,\lceil N/2\rceil+1,\ldots,N\}$
we have

\[
\frac{N-s+1}{2}\min_{k\in\mathcal{N}}\gamma_{k}\eta_{k}\geq\frac{N}{4}\min_{k\in\mathcal{N}}\frac{6D_{\mathcal{X}}^{2}}{k}\geq\frac{3}{2}D_{\mathcal{X}}^{2},
\]
and for $\mathcal{B}\cup\mathcal{N}=\{\lceil N/2\rceil,\lceil N/2\rceil+1,\ldots,N\}$
we have 

\begin{eqnarray}
D_{\mathcal{X}}^{2}+\frac{1}{2}\sum_{k\in\mathcal{B}}\gamma_{k}^{2}L_{f}^{2}+\frac{1}{2}\sum_{k\in\mathcal{N}}\gamma_{k}^{2}L_{g,\mathcal{X}}^{2} & \leq & D_{\mathcal{X}}^{2}+\frac{1}{2}\sum_{k\in\mathcal{B}}\frac{D_{\mathcal{X}}^{2}L_{f}^{2}}{k(L_{f}+L_{g,\mathcal{X}})^{2}}+\frac{1}{2}\sum_{k\in\mathcal{N}}\frac{D_{\mathcal{X}}^{2}L_{g,\mathcal{X}}^{2}}{k(L_{f}+L_{g,\mathcal{X}})^{2}}\nonumber \\
 & \leq & D_{\mathcal{X}}^{2}+\frac{1}{2}\sum_{k=\lceil N/2\rceil}^{N}\frac{D_{\mathcal{X}}^{2}}{\lceil N/2\rceil}\leq\frac{3}{2}D_{\mathcal{X}}^{2}.\label{willbeusedlater}
\end{eqnarray}
By Lemma \ref{lemforsuff}, $\mathcal{B}\neq\emptyset$ and $\overline{x}_{N,s}$
is well-defined. 
\end{proof}
The main convergence properties of Algorithm \ref{CSAinSIP} are established
next in Proposition \ref{converg}.
\begin{prop}
\label{converg} Suppose that $\{\gamma_{k}\}_{k\geq1}$ and $\{\eta_{k}\}_{k\geq1}$
are chosen such that (\ref{suff}) holds, and let $\left\{ x_{k}\right\} _{k\geq1}$
be produced by Algorithm \ref{CSAinSIP}. Then, for any $1\leq s\leq N$
we have
\begin{eqnarray}
f(\overline{x}_{N,s})-f(x^{*})\leq\frac{2D_{\mathcal{X}}^{2}+\sum_{k\in\mathcal{B}}\gamma_{k}^{2}L_{f}^{2}+\sum_{k\in\mathcal{N}}\gamma_{k}^{2}L_{g,\mathcal{X}}^{2}}{(N-s+1)\min_{k\in\mathcal{B}}\gamma_{k}},\label{convg1}
\end{eqnarray}
and
\begin{eqnarray}
G(\overline{x}_{N,s})\leq\frac{\sum_{k\in\mathcal{B}}\gamma_{k}(\eta_{k}+\varepsilon_{k})}{\sum_{k\in\mathcal{B}}\gamma_{k}}.\label{convg2}
\end{eqnarray}
\end{prop}

\begin{proof}
First, we show that (\ref{convg1}) holds. By Lemma \ref{lemforsuff},
if $\sum_{k\in\mathcal{B}}\gamma_{k}\langle f'(x_{k}),x_{k}-x^{*}\rangle<0$,
then by the convexity of $f$ and the definition of $\overline{x}_{N,s}$,
we have

\begin{equation}
f(\overline{x}_{N,s})-f(x^{*})\leq\frac{\sum_{k\in\mathcal{B}}\gamma_{k}(f(x_{k})-f(x^{*}))}{\sum_{k\in\mathcal{B}}\gamma_{k}}\leq\frac{\sum_{k\in\mathcal{B}}\gamma_{k}\langle f'(x_{k}),x_{k}-x^{*}\rangle}{\sum_{k\in\mathcal{B}}\gamma_{k}}<0.\label{case1}
\end{equation}
If $|\mathcal{B}|\geq(N-s+1)/2$, we have $\sum_{k\in\mathcal{B}}\gamma_{k}\geq|\mathcal{B}|\min_{k\in\mathcal{B}}\gamma_{k}\geq\frac{N-s+1}{2}\min_{k\in\mathcal{B}}\gamma_{k}$.
By fixing $x=x^{*}$ in (\ref{recursion}), it follows from the definition
of $\overline{x}_{N,s}$ and the convexity of $f$ that
\begin{align*}
\sum_{k\in\mathcal{N}}\gamma_{k}\eta_{k}+\sum_{k\in\mathcal{B}}\gamma_{k}[f(\overline{x}_{N,s})-f(x^{*})] & \leq\sum_{k\in\mathcal{N}}\gamma_{k}\eta_{k}+\sum_{k\in\mathcal{B}}[\gamma_{k}(f(x_{k})-f(x^{*}))]\\
 & \leq D_{\mathcal{X}}^{2}+\frac{1}{2}\sum_{k\in\mathcal{B}}\gamma_{k}^{2}L_{f}^{2}+\frac{1}{2}\sum_{k\in\mathcal{N}}\gamma_{k}^{2}L_{g,\mathcal{X}}^{2}.
\end{align*}
Noticing $\sum_{k\in\mathcal{N}}\gamma_{k}\eta_{k}\geq0$, it follows
that
\begin{align}
f(\overline{x}_{N,s})-f(x^{*})\leq\frac{2D_{\mathcal{X}}^{2}+\sum_{k\in\mathcal{B}}\gamma_{k}^{2}L_{f}^{2}+\sum_{k\in\mathcal{N}}\gamma_{k}^{2}L_{g,\mathcal{X}}^{2}}{(N-s+1)\min_{k\in\mathcal{B}}\gamma_{k}}.\label{case2}
\end{align}
We then have (\ref{convg1}) by the two inequalities (\ref{case1})
and (\ref{case2}). Next we prove (\ref{convg2}). For any $k\in\mathcal{B}$,
we have $g(x_{k},\delta_{k})\leq\eta_{k}$ and so $G(x_{k})\leq\eta_{k}+\varepsilon_{k}$.
From the definition of $\overline{x}_{N,s}$ and the convexity of
$G\left(\cdot\right)$, we then have
\[
G(\overline{x}_{N,s})\leq\frac{\sum_{k\in\mathcal{B}}\gamma_{k}G(x_{k})}{\sum_{k\in\mathcal{B}}\gamma_{k}}\leq\frac{\sum_{k\in\mathcal{B}}\gamma_{k}(\eta_{k}+\varepsilon_{k})}{\sum_{k\in\mathcal{B}}\gamma_{k}}.
\]
\end{proof}
Now we may prove the error bounds for inexact CSA for the generally
convex case.
\begin{proof}[Proof of Theorem \ref{varistepresult}]
 The bound on the optimality gap comes from (\ref{convg1}). Recall
(\ref{willbeusedlater}), i.e. we have $2D_{\mathcal{X}}^{2}+\sum_{k\in\mathcal{B}}\gamma_{k}^{2}L_{f}^{2}+\sum_{k\in\mathcal{N}}\gamma_{k}^{2}L_{g,\mathcal{X}}^{2}\leq3D_{\mathcal{X}}^{2}$.
From $s=\lceil\frac{N}{2}\rceil\leq\frac{N}{2}+1$, $\gamma_{k}=\frac{D_{\mathcal{X}}}{\sqrt{k}(L_{f}+L_{g,\mathcal{X}})}$,
and $\mathcal{B}\subset\{s,\ldots,N\}$, we obtain $(N-s+1)\min_{k\in\mathcal{B}}\gamma_{k}\geq\frac{N}{2}\frac{D_{\mathcal{X}}}{\sqrt{N}(L_{f}+L_{g,\mathcal{X}})}=\frac{\sqrt{N}D_{\mathcal{X}}}{2(L_{f}+L_{g,\mathcal{X}})}$.
It then follows from (\ref{convg1}) that 
\begin{align*}
f(\overline{x}_{N,s})-f(x^{*})\leq\frac{6D_{\mathcal{X}}(L_{f}+L_{g,\mathcal{X}})}{\sqrt{N}}.
\end{align*}
The bound on constraint violation is by (\ref{convg2}). For any $k\in\mathcal{B}\subset\{\lceil\frac{N}{2}\rceil,\lceil\frac{N}{2}\rceil+1,\ldots,N\}$,
we have $\gamma_{k}\eta_{k}=\frac{6D_{\mathcal{X}}^{2}}{k}\leq\frac{12D_{\mathcal{X}}^{2}}{N}$
and $\gamma_{k}\geq\frac{D_{\mathcal{X}}}{\sqrt{N}(L_{f}+L_{g,\mathcal{X}})}$.
It then follows from (\ref{convg2}) that 
\[
G(\overline{x}_{N,s})\leq\frac{12D_{\mathcal{X}}(L_{f}+L_{g,\mathcal{X}})}{\sqrt{N}}+\frac{\sum_{k\in\mathcal{B}}\varepsilon_{k}/\sqrt{k}}{\sum_{k\in\mathcal{B}}1/\sqrt{k}}.
\]
\end{proof}

\subsubsection{\label{subsec:Strongly-Convex-Case}Strongly Convex Case}

Now we consider the general error bounds for the strongly convex case
(Theorem \ref{strongconvexresult}). The following lemma will be used
in subsequent results, its proof is straightforward and so the details
are skipped. We remind the reader that $\{A_{k}\}_{k\geq1}$ is defined
in Section \ref{generalresults}.
\begin{lem}
\label{triangle} For all $k\geq1$, let $a_{k}\in(0,1]$ and $A_{k}>0$.
If sequences $\{\triangle_{k}\}_{k\geq1}$ and $\{B_{k}\}_{k\geq1}$
satisfy $\triangle_{k+1}\leq(1-a_{k})\triangle_{k}+B_{k}$ for all
$k\geq1$, then for any $1\leq s\leq k$ we have
\[
\frac{\triangle_{k+1}}{A_{k}}\leq\frac{(1-a_{s})\triangle_{s}}{A_{s}}+\sum_{i=s}^{k}\frac{B_{i}}{A_{i}}.
\]
\end{lem}

We remind the reader that $\{\rho_{k}\}_{k\geq1}$ in the next result
is originally defined in Section \ref{generalresults}.
\begin{prop}
Suppose Assumption \ref{strongconvex_and_L} holds. Choose stepsizes
$\{\gamma_{k}\}_{k\geq1}$, tolerances $\{\eta_{k}\}_{k\geq1}$, and
$1\leq s\leq N$. Let $\left\{ x_{k}\right\} _{k\geq1}$ be produced
according to Algorithm \ref{CSAinSIP}, then
\begin{align}
 & \sum_{k\in\mathcal{N}}\rho_{k}(\eta_{k}-g(x,\delta_{k}))+\sum_{k\in\mathcal{B}}\rho_{k}(f(x_{k})-f(x))\nonumber \\
\leq & \frac{(1-a_{s})V(x_{s},x)}{A_{s}}+\frac{1}{2}\sum_{k\in\mathcal{B}}\rho_{k}\gamma_{k}\|f'(x_{k})\|^{2}+\frac{1}{2}\sum_{k\in\mathcal{N}}\rho_{k}\gamma_{k}\|g'(x_{k},\delta_{k})\|^{2},\label{recursforstrong}
\end{align}
for all $x\in\mathcal{X}$. 
\end{prop}

\begin{proof}
Consider an iteration $s\leq k\leq N$. If $k\in\mathcal{B}$, then
by Lemma \ref{block} and Assumption \ref{strongconvex_and_L}, we
have 
\begin{align*}
V(x_{k+1},x) & \leq V(x_{k},x)-\gamma_{k}\langle f'(x_{k}),x_{k}-x\rangle+\frac{1}{2}\gamma_{k}^{2}\|f'(x_{k})\|^{2}\\
 & \leq V(x_{k},x)-\gamma_{k}[f(x_{k})-f(x)+\frac{\mu_{f}}{2}\|x_{k}-x\|^{2}]+\frac{1}{2}\gamma_{k}^{2}\|f'(x_{k})\|^{2}\\
 & \leq(1-\frac{\mu_{f}\gamma_{k}}{L})V(x_{k},x)-\gamma_{k}[f(x_{k})-f(x)]+\frac{1}{2}\gamma_{k}^{2}\|f'(x_{k})\|^{2}.
\end{align*}
Similarly, for $k\in\mathcal{N}$, by Lemma \ref{block} and Assumption
\ref{strongconvex_and_L}, we have 
\begin{align*}
V(x_{k+1},x) & \leq V(x_{k},x)-\gamma_{k}\langle g'(x_{k},\delta_{k}),x_{k}-x\rangle+\frac{1}{2}\gamma_{k}^{2}\|g'(x_{k},\delta_{k})\|^{2}\\
 & \leq V(x_{k},x)-\gamma_{k}[g(x_{k},\delta_{k})-g(x,\delta_{k})+\frac{\mu_{g}}{2}\|x_{k}-x\|^{2}]+\frac{1}{2}\gamma_{k}^{2}\|g'(x_{k},\delta_{k})\|^{2}\\
 & \leq(1-\frac{\mu_{g}\gamma_{k}}{L})V(x_{k},x)-\gamma_{k}[\eta_{k}-g(x,\delta_{k})]+\frac{1}{2}\gamma_{k}^{2}\|g'(x_{k},\delta_{k})\|^{2}.
\end{align*}
Invoking Lemma \ref{triangle}, we then obtain
\begin{align*}
0\leq\frac{V(x_{N+1},x)}{A_{N}}\leq & \frac{(1-a_{s})V(x_{s},x)}{A_{s}}-\big[\sum_{k\in\mathcal{N}}\frac{\gamma_{k}}{A_{k}}(\eta_{k}-g(x,\delta_{k}))+\sum_{k\in\mathcal{B}}\frac{\gamma_{k}}{A_{k}}(f(x_{k})-f(x))\big]\\
 & +\frac{1}{2}\sum_{k\in\mathcal{B}}\frac{\gamma_{k}^{2}}{A_{k}}\|f'(x_{k})\|^{2}+\frac{1}{2}\sum_{k\in\mathcal{N}}\frac{\gamma_{k}^{2}}{A_{k}}\|g'(x_{k},\delta_{k})\|^{2}.
\end{align*}
Rearranging the terms in the above inequality and recalling the definition
of $\rho_{k}$, we arrive at (\ref{recursforstrong}). 
\end{proof}
The following result provides a sufficient condition for $\overline{x}_{N,s}$
to be well-defined. 
\begin{lem}
\label{Lemsuffstrong} Suppose Assumption \ref{strongconvex_and_L}
and the condition
\begin{eqnarray}
\min_{\left\{ \mathcal{A}\subset I:\left|\mathcal{A}\right|=\left\lceil (N-s+1)/2\right\rceil \right\} }\sum_{k\in\mathcal{A}}\rho_{k}\eta_{k}>\frac{(1-a_{s})D_{\mathcal{X}}^{2}}{A_{s}}+\frac{1}{2}\sum_{k\in\mathcal{B}}\rho_{k}\gamma_{k}L_{f}^{2}+\frac{1}{2}\sum_{k\in\mathcal{N}}\rho_{k}\gamma_{k}L_{g,\mathcal{X}}^{2}\label{suffstrong}
\end{eqnarray}
hold. Then, $\mathcal{B}\neq\emptyset$, i.e., $\overline{x}_{N,s}$
is well-defined. Furthermore, either (i) $|\mathcal{B}|\geq(N-s+1)/2$
or (ii) $\sum_{k\in\mathcal{B}}\rho_{k}(f(x_{k})-f(x^{*}))<0$ holds. 
\end{lem}

\begin{proof}
By fixing $x=x^{*}$ in (\ref{recursforstrong}), we obtain

\[
\sum_{k\in\mathcal{N}}\rho_{k}(\eta_{k}-g(x^{*},\delta_{k}))+\sum_{k\in\mathcal{B}}\rho_{k}(f(x_{k})-f(x^{*}))\leq\frac{(1-a_{s})V(x_{s},x^{*})}{A_{s}}+\frac{1}{2}\sum_{k\in\mathcal{B}}\rho_{k}\gamma_{k}L_{f}^{2}+\frac{1}{2}\sum_{k\in\mathcal{N}}\rho_{k}\gamma_{k}L_{g,\mathcal{X}}^{2}.
\]
If $\sum_{k\in\mathcal{B}}\rho_{k}(f(x_{k})-f(x^{*}))\geq0$, noticing
$g(x^{*},\delta_{k})\leq G(x^{*})\leq0$, we have 
\begin{eqnarray}
\sum_{k\in\mathcal{N}}\rho_{k}\eta_{k}\leq\frac{(1-a_{s})V(x_{s},x^{*})}{A_{s}}+\frac{1}{2}\sum_{k\in\mathcal{B}}\rho_{k}\gamma_{k}L_{f}^{2}+\frac{1}{2}\sum_{k\in\mathcal{N}}\rho_{k}\gamma_{k}L_{g,\mathcal{X}}^{2}.\label{contr-1}
\end{eqnarray}
Suppose that $|\mathcal{B}|<(N-s+1)/2$, i.e., $|\mathcal{N}|\geq(N-s+1)/2$.
Then, by assumption we have
\[
\sum_{k\in\mathcal{N}}\rho_{k}\eta_{k}\geq\min_{\left\{ \mathcal{A}\subset I:\left|\mathcal{A}\right|=\left\lceil (N-s+1)/2\right\rceil \right\} }\sum_{k\in\mathcal{A}}\rho_{k}\eta_{k}>\frac{(1-a_{s})D_{\mathcal{X}}^{2}}{A_{s}}+\frac{1}{2}\sum_{k\in\mathcal{B}}\rho_{k}\gamma_{k}L_{f}^{2}+\frac{1}{2}\sum_{k\in\mathcal{N}}\rho_{k}\gamma_{k}L_{g,\mathcal{X}}^{2},
\]
which contradicts (\ref{contr-1}). Thus, condition (i) holds. Alternatively,
if $\sum_{k\in\mathcal{B}}\rho_{k}(f(x_{k})-f(x^{*}))<0$ then condition
(ii) holds. 
\end{proof}
Based on the above lemma, we may now prove Lemma \ref{nonemptyBstronglyconvex}. 
\begin{proof}[Proof of Lemma \ref{nonemptyBstronglyconvex}]
 From the selections of $\{\eta_{k}\}_{k\geq1}$, $\{\gamma_{k}\}_{k\geq1}$,
and $s$ in (\ref{stepsizeinstrongconvex}), we have for $k\in\{1,2,\ldots,N\}$,
$a_{k}=\frac{2}{k+1}$, $A_{k}=\frac{2}{(k+1)k}$, and 

\[
\rho_{k}=\frac{\gamma_{k}}{A_{k}}=\left\{ \begin{array}{l}
\frac{Lk}{\mu_{f}},\quad k\in\mathcal{B},\\
\frac{Lk}{\mu_{g}},\quad k\in\mathcal{N}.
\end{array}\right.
\]
Specifically, $a_{s}=a_{1}=1$, and 
\[
\frac{1}{2}\rho_{k}\gamma_{k}=\frac{\gamma_{k}^{2}}{2A_{k}}=\left\{ \begin{array}{l}
\frac{L^{2}}{\mu_{f}^{2}}\frac{k}{k+1}\leq\frac{L^{2}}{\mu_{f}^{2}},\quad k\in\mathcal{B},\\
\frac{L^{2}}{\mu_{g}^{2}}\frac{k}{k+1}\leq\frac{L^{2}}{\mu_{g}^{2}},\quad k\in\mathcal{N},
\end{array}\right.
\]
which implies that 
\begin{equation}
\frac{(1-a_{s})D_{\mathcal{X}}^{2}}{A_{s}}+\frac{1}{2}\sum_{k\in\mathcal{B}}\rho_{k}\gamma_{k}L_{f}^{2}+\frac{1}{2}\sum_{k\in\mathcal{N}}\rho_{k}\gamma_{k}L_{g,\mathcal{X}}^{2}\leq L^{2}(|\mathcal{B}|\frac{L_{f}^{2}}{\mu_{f}^{2}}+|\mathcal{N}|\frac{L_{g,\mathcal{X}}^{2}}{\mu_{g}^{2}})\leq NL^{2}\max\left\{ \frac{L_{f}^{2}}{\mu_{f}^{2}},\frac{L_{g,\mathcal{X}}^{2}}{\mu_{g}^{2}}\right\} .\label{usedlaterstronglyconvex}
\end{equation}
Also, we have

\[
\begin{array}{rcl}
\min_{\left\{ \mathcal{A}\subset I:\left|\mathcal{A}\right|=\left\lceil (N-s+1)/2\right\rceil \right\} }\sum_{k\in\mathcal{A}}\rho_{k}\eta_{k} & \geq & \sum_{k=1}^{\left\lceil N/2\right\rceil }\frac{Lk}{\max\left\{ \mu_{f},\mu_{g}\right\} }.\frac{8L}{N}\max\left\{ \mu_{f},\mu_{g}\right\} \max\left\{ \frac{L_{f}^{2}}{\mu_{f}^{2}},\frac{L_{g,\mathcal{X}}^{2}}{\mu_{g}^{2}}\right\} \\
 & \geq & (N+1)L^{2}\max\left\{ \frac{L_{f}^{2}}{\mu_{f}^{2}},\frac{L_{g,\mathcal{X}}^{2}}{\mu_{g}^{2}}\right\} .
\end{array}
\]
By Lemma \ref{Lemsuffstrong}, we have $\mathcal{B}\neq\emptyset$,
i.e., $\overline{x}_{N,s}$ is well-defined. 
\end{proof}
Before we prove Theorem \ref{strongconvexresult}, we establish the
main convergence properties of Algorithm \ref{CSAinSIP} in the following
proposition.
\begin{prop}
\label{convergstrong} Suppose Assumption \ref{strongconvex_and_L}
holds, and suppose that $\{\gamma_{k}\}_{k\geq1}$ and $\{\eta_{k}\}_{k\geq1}$
are chosen such that (\ref{suffstrong}) holds. Let $\left\{ x_{k}\right\} _{k\geq1}$
be generated according to Algorithm \ref{CSAinSIP}. Then, for any
$1\leq s\leq N$ we have
\begin{eqnarray}
f(\overline{x}_{N,s})-f(x^{*})\leq\frac{2(1-a_{s})D_{\mathcal{X}}^{2}/A_{s}+\sum_{k\in\mathcal{B}}\rho_{k}\gamma_{k}L_{f}^{2}+\sum_{k\in\mathcal{N}}\rho_{k}\gamma_{k}L_{g,\mathcal{X}}^{2}}{2\min_{\left\{ \mathcal{A}\subset I:\left|\mathcal{A}\right|=\left\lceil (N-s+1)/2\right\rceil \right\} }\sum_{k\in\mathcal{A}}\rho_{k}},\label{convgstr1}
\end{eqnarray}
and
\begin{eqnarray}
G(\overline{x}_{N,s})\leq\frac{\sum_{k\in\mathcal{B}}\rho_{k}(\eta_{k}+\varepsilon_{k})}{\sum_{k\in\mathcal{B}}\rho_{k}}.\label{convgstr2}
\end{eqnarray}
\end{prop}

\begin{proof}
We first show that (\ref{convgstr1}) holds. By Lemma \ref{Lemsuffstrong},
we have two cases. If $\sum_{k\in\mathcal{B}}\rho_{k}(f(x_{k})-f(x^{*}))<0$
holds, using the convexity of $f$ and the definition of $\overline{x}_{N,s}$,
we obtain $f(\overline{x}_{N,s})-f(x^{*})<0$ which implies (\ref{convgstr1}).
If $|\mathcal{B}|\geq(N-s+1)/2$, then we have $\sum_{k\in\mathcal{B}}\rho_{k}\geq\min_{\left\{ \mathcal{A}\subset I:\left|\mathcal{A}\right|=\left\lceil (N-s+1)/2\right\rceil \right\} }\sum_{k\in\mathcal{A}}\rho_{k}$.
Take $x=x^{*}$ in (\ref{recursforstrong}), from Assumptions\textbf{\textcolor{black}{{}
A2}}, \textbf{\textcolor{black}{A4}}, the definition of $\overline{x}_{N,s}$,
and the fact that $g(x^{*},\delta_{k})\leq G(x^{*})\leq0$, we have

\[
\begin{array}{rcl}
\sum_{k\in\mathcal{N}}\rho_{k}\eta_{k}+\sum_{k\in\mathcal{B}}\rho_{k}[f(\overline{x}_{N,s})-f(x^{*})] & \leq & \sum_{k\in\mathcal{N}}\rho_{k}\eta_{k}+\sum_{k\in\mathcal{B}}[\rho_{k}(f(x_{k})-f(x^{*}))]\\
 & \leq & (1-a_{s})D_{\mathcal{X}}^{2}/A_{s}+\frac{1}{2}\sum_{k\in\mathcal{B}}\rho_{k}\gamma_{k}L_{f}^{2}+\frac{1}{2}\sum_{k\in\mathcal{N}}\rho_{k}\gamma_{k}L_{g,\mathcal{X}}^{2}.
\end{array}
\]
Noticing $\sum_{k\in\mathcal{N}}\rho_{k}\eta_{k}\geq0$, it follows
that (\ref{convgstr1}) holds.

Next we prove (\ref{convgstr2}). For any $k\in\mathcal{B}$, we have
$g(x_{k},\delta_{k})\leq\eta_{k}$ by definition. Then, for any $k\in\mathcal{B}$
we must have $G(x_{k})\leq\eta_{k}+\varepsilon_{k}$. From the definition
of $\overline{x}_{N,s}$, and the convexity of $G$, we obtain
\[
G(\overline{x}_{N,s})\leq\frac{\sum_{k\in\mathcal{B}}\rho_{k}G(x_{k})}{\sum_{k\in\mathcal{B}}\rho_{k}}\leq\frac{\sum_{k\in\mathcal{B}}\rho_{k}(\eta_{k}+\varepsilon_{k})}{\sum_{k\in\mathcal{B}}\rho_{k}}.
\]
\end{proof}
We now have the machinery in place to prove the error bound for inexact
CSA in the strongly convex case. 
\begin{proof}[Proof of Theorem \ref{strongconvexresult}]
 We bound the optimality gap by (\ref{convgstr1}) as follows. Recall
(\ref{usedlaterstronglyconvex}), we have 

\[
2(1-a_{s})D_{\mathcal{X}}^{2}/A_{s}+\sum_{k\in\mathcal{B}}\rho_{k}\gamma_{k}L_{f}^{2}+\sum_{k\in\mathcal{N}}\rho_{k}\gamma_{k}L_{g,\mathcal{X}}^{2}\leq2NL^{2}\max\left\{ \frac{L_{f}^{2}}{\mu_{f}^{2}},\frac{L_{g,\mathcal{X}}^{2}}{\mu_{g}^{2}}\right\} .
\]
Further, we have $\min_{\left\{ \mathcal{A}\subset I:\left|\mathcal{A}\right|=\left\lceil (N-s+1)/2\right\rceil \right\} }\sum_{k\in\mathcal{A}}\rho_{k}\geq\sum_{k=1}^{\left\lceil N/2\right\rceil }\frac{Lk}{\max\left\{ \mu_{f},\mu_{g}\right\} }\geq\frac{LN(N+1)}{8\max\left\{ \mu_{f},\mu_{g}\right\} }$.
It then follows from (\ref{convgstr1}) that 
\begin{align*}
f(\overline{x}_{N,s})-f(x^{*})\leq\frac{8L}{N+1}\max\left\{ \mu_{f},\mu_{g}\right\} \max\left\{ \frac{L_{f}^{2}}{\mu_{f}^{2}},\frac{L_{g,\mathcal{X}}^{2}}{\mu_{g}^{2}}\right\} .
\end{align*}
Next, we bound the constraint violation by (\ref{convgstr2}). Noticing
that $\eta_{k}=\frac{8L}{N}\max\left\{ \mu_{f},\mu_{g}\right\} \max\left\{ \frac{L_{f}^{2}}{\mu_{f}^{2}},\frac{L_{g,\mathcal{X}}^{2}}{\mu_{g}^{2}}\right\} $
is a constant, it immediately follows from (\ref{convgstr2}) that
\[
G(\overline{x}_{N,s})\leq\frac{8L}{N}\max\left\{ \mu_{f},\mu_{g}\right\} \max\left\{ \frac{L_{f}^{2}}{\mu_{f}^{2}},\frac{L_{g,\mathcal{X}}^{2}}{\mu_{g}^{2}}\right\} +\frac{\sum_{k\in\mathcal{B}}k\,\varepsilon_{k}}{\sum_{k\in\mathcal{B}}k}.
\]
\end{proof}

\subsection{\label{subsec:Fixed-Sampling} CSA with Fixed Sampling}

In this subsection we develop the proofs for our fixed sampling scheme.
At each iteration $k$, $x_{k}$ is fixed, and we face the cut generation
Problem (\ref{cut}) which can be written in epigraph form (where
the index $k$ is omitted): 
\begin{eqnarray}
\min_{y\in\mathbb{R}}\Big\{ y:y\geq g(x,\delta),\forall\delta\in\Delta\Big\}.\label{opti}
\end{eqnarray}

We repeat the definition of uniform level-set bound (ULB) from \cite{esfahani2015performance}
as follows.
\begin{defn}
\cite[Definition 3.1]{esfahani2015performance}\label{ULB} For fixed
$x\in\mathcal{X}$, the tail probability of the worst-case violation
is the function $p:\mathbb{R}_{+}\rightarrow[0,1]$ defined by $p(\alpha):=Q\{\delta\in\Delta:g(x,\delta)>G(x)-\alpha\}$.
We call $h:[0,1]\rightarrow\mathbb{R}_{+}$ a uniform level-set bound
(ULB) of $p$ if for all $\varepsilon\in[0,1]$, $h(\varepsilon)\geq\sup\{\kappa\in\mathbb{R}_{+}:p(\kappa)\leq\varepsilon\}$.
\end{defn}

Let $\delta^{(1)},\delta^{(2)},\ldots,\delta^{(M)}$ be i.i.d. samples
generated according to a probability distribution $Q$. The sampled
problem derived from Problem (\ref{opti}) is 
\begin{eqnarray}
\min_{y\in\mathbb{R}}\Big\{ y:y\geq g(x,\delta^{(i)}),\forall i=1,2,\ldots,M\Big\}\label{sampled}
\end{eqnarray}
which is equivalent to $\max_{i=1,2,\ldots,M}g(x,\delta^{(i)})$.

Let $\widehat{g}_{M}(x)$ be the unique solution of Problem (\ref{sampled}).
This optimal solution $\widehat{g}_{M}(x)$ is a random variable that
depends on the samples $\delta^{(1)},\delta^{(2)},\ldots,\delta^{(M)}$.
As a direct application of Theorem 3.6 in \cite{esfahani2015performance},
we have the following key result. 
\begin{prop}
\label{keyresult} Consider the Problems (\ref{opti}) and (\ref{sampled})
for fixed $x\in\mathcal{X}$ with the associated optimal values $G(x)$
and $\widehat{g}_{M}(x)$, respectively. Given a ULB $h$ and $\varepsilon$,
$\beta$ in $[0,1]$, for all $M\geq M(\varepsilon,\beta)$, we have
$Q^{M}\{G(x)-\widehat{g}_{M}(x)\in[0,h(\varepsilon)]\}\geq1-\beta$.
\end{prop}

From Proposition \ref{keyresult}, we see that for fixed $x\in\mathcal{X}$
the gap between $\widehat{g}_{M}(x)$ and $G(x)$ is effectively quantified
by a ULB $h(\varepsilon)$. To control the behavior of $h(\varepsilon)$
as $\varepsilon\rightarrow0$, we require more structure on the probability
distribution $Q$ on $\Delta$, which is imposed in Assumption \ref{strictlyincr}.
The next result is based on Assumption \ref{strictlyincr}.
\begin{prop}
\cite[Proposition 3.8]{esfahani2015performance}\label{ULBexample}
Under Assumption \ref{strictlyincr}, the function $h(\varepsilon):=L_{g,\Delta}\varphi^{-1}(\varepsilon)$
is a ULB, where $\varphi^{-1}$ is the inverse of $\varphi$. 
\end{prop}

From Propositions \ref{keyresult} and \ref{ULBexample}, we obtain
the following bound in probability. 
\begin{prop}
\label{keyrandcut} Suppose Assumption \ref{strictlyincr} holds.
Given $\epsilon>0$ and $\beta\in(0,1)$, for $M\geq M(\varphi(\frac{\epsilon}{L_{g,\Delta}}),\beta)$
i.i.d. samples from $Q$, we have $Q^{M}\{G(x)-\max_{1\leq i\leq M}g(x,\delta^{(i)})\leq\epsilon\}\geq1-\beta$.
\end{prop}

Now we can estimate the empirical constraint violation for the fixed
sampling scheme. 
\begin{proof}[Proof of Proposition \ref{keyrandcut-inexpectation}]
 From Proposition \ref{keyrandcut}, we have 

\[
Q^{M}\left\{ G(x)-\max_{1\leq i\leq M}g(x,\delta^{(i)})\in[0,\frac{\epsilon}{2}]\right\} \geq1-\frac{\epsilon}{2(\overline{M}-\underline{M})}.
\]
Therefore, we have 

\[
\begin{array}{rcl}
\mathbb{E}_{Q^{M}}\left[\max_{1\leq i\leq M}g(x,\delta^{(i)})\right] & \geq & (G(x)-\frac{\epsilon}{2})(1-\frac{\epsilon}{2(\overline{M}-\underline{M})})+\underline{M}\frac{\epsilon}{2(\overline{M}-\underline{M})}\\
 & \geq & G(x)-\frac{\epsilon}{2}-\frac{\epsilon(G(x)-\underline{M})}{2(\overline{M}-\underline{M})}\\
 & \geq & G(x)-\epsilon.
\end{array}
\]
\end{proof}
Next, we give the proof for Theorem \ref{randvaristep} (for the generally
convex case under the fixed sampling scheme). The proof uses Proposition
\ref{keyrandcut-inexpectation} to control the error terms in our
general inexact CSA analysis.
\begin{proof}[Proof of Theorem \ref{randvaristep}]
 From Proposition \ref{keyrandcut-inexpectation}, we have $\mathbb{E}_{Q^{M_{k}}}\left[\varepsilon_{k}\right]\leq\frac{(L_{f}+L_{g,\mathcal{X}})D_{\mathcal{X}}}{\sqrt{k}}$
for all $k\geq1$. For $k\in\mathcal{B}$ with $k>\frac{N}{2}$, we
have $\mathbb{E}_{Q^{M_{k}}}\left[\frac{\varepsilon_{k}}{\sqrt{k}}\right]\leq\frac{2D_{\mathcal{X}}(L_{f}+L_{g,\mathcal{X}})}{N}$.
Moreover, $\frac{1}{\sqrt{k}}\geq\frac{1}{\sqrt{N}}$. Thus, from
independence of samples, we have

\[
\mathbb{E}_{\mathcal{Q}}\left[\frac{\sum_{k\in\mathcal{B}}\varepsilon_{k}/\sqrt{k}}{\sum_{k\in\mathcal{B}}1/\sqrt{k}}\right]\leq\frac{2D_{\mathcal{X}}(L_{f}+L_{g,\mathcal{X}})}{\sqrt{N}}.
\]
Subsequently, Theorem \ref{varistepresult} gives $\mathbb{E}_{\mathcal{Q}}\left[G(\overline{x}_{N,s})\right]\leq14D_{\mathcal{X}}(L_{f}+L_{g,\mathcal{X}})/\sqrt{N}$.
\end{proof}
The proof of Theorem \ref{randstrong} (for the strongly convex case)
is as follows. 
\begin{proof}[Proof of Theorem \ref{randstrong}]
 From Proposition \ref{keyrandcut-inexpectation}, we have $\mathbb{E}_{Q^{M_{k}}}\left[\varepsilon_{k}\right]\leq\frac{L}{N}\max\left\{ \mu_{f},\mu_{g}\right\} \max\left\{ \frac{L_{f}^{2}}{\mu_{f}^{2}},\frac{L_{g,\mathcal{X}}^{2}}{\mu_{g}^{2}}\right\} $.
It follows that

\[
\mathbb{E}_{\mathcal{Q}}\left[\frac{\sum_{k\in\mathcal{B}}k\,\varepsilon_{k}}{\sum_{k\in\mathcal{B}}k}\right]\leq\frac{L}{N}\max\left\{ \mu_{f},\mu_{g}\right\} \max\left\{ \frac{L_{f}^{2}}{\mu_{f}^{2}},\frac{L_{g,\mathcal{X}}^{2}}{\mu_{g}^{2}}\right\} .
\]
Therefore, from Theorem \ref{strongconvexresult}, we arrive at the
inequality $\mathbb{E}_{\mathcal{Q}}\left[G(\overline{x}_{N,s})\right]\leq9L\max\left\{ \mu_{f},\mu_{g}\right\} \max\left\{ \frac{L_{f}^{2}}{\mu_{f}^{2}},\frac{L_{g,\mathcal{X}}^{2}}{\mu_{g}^{2}}\right\} /N$.
\end{proof}

\subsection{\label{subsec:Adaptive-Sampling} CSA with Adaptive Sampling}

This subsection considers the adaptive sampling scheme. First, we
need to prove two prerequisite Lemmas \ref{transform} and \ref{existenceofmaximizer}.
Lemma \ref{transform} establishes an equivalence between the nonlinear
finite-dimensional optimization problem $\max_{\delta\in\Delta}g(x,\delta)$
and an infinite-dimensional linear optimization problem $\max_{\phi\in\mathcal{P}(\Delta)}\mathbb{E}_{\widetilde{\delta}\sim\phi}\left[g\left(x,\widetilde{\delta}\right)\right]$. 
\begin{proof}[Proof of Lemma \ref{transform}]
 The existence of a maximizer $\delta^{\ast}(x)\in\arg\max_{\delta\in\Delta}g(x,\,\delta)$
can be guaranteed by Assumptions\textbf{\textcolor{black}{{} A3}}, \textbf{\textcolor{black}{A5}}.
On one hand, for any $\phi\in\mathcal{P}(\Delta)$,

\[
\mathbb{E}_{\delta\sim\phi}\left[g\left(x,\delta\right)\right]=\int_{\Delta}g(x,\,\delta)\phi(d\delta)\leq\int_{\Delta}\max_{\delta\in\Delta}g(x,\delta)\phi(d\delta)=\max_{\delta\in\Delta}g(x,\delta).
\]
Since $\phi$ is arbitrary, we have $\max_{\delta\in\Delta}g(x,\delta)\geq\max_{\phi\in\mathcal{P}(\Delta)}\mathbb{E}_{\widetilde{\delta}\sim\phi}\left[g\left(x,\widetilde{\delta}\right)\right]$.
On the other hand, we can put all mass of $\phi$ on $\delta^{\ast}(x)$,
i.e., the Dirac measure $\phi=\delta_{\delta^{\ast}(x)}$, thus $\mathbb{E}_{\widetilde{\delta}\sim\delta_{\delta^{\ast}(x)}}\left[g\left(x,\,\widetilde{\delta}\right)\right]=\max_{\delta\in\Delta}g(x,\delta)$,
which implies $\max_{\delta\in\Delta}g(x,\delta)\leq\max_{\phi\in\mathcal{P}(\Delta)}\mathbb{E}_{\widetilde{\delta}\sim\phi}\left[g\left(x,\widetilde{\delta}\right)\right]$.
\end{proof}
Lemma \ref{existenceofmaximizer} justifies the existence of a solution
of the regularized cut generation Problem (\ref{regulcutgen}), and
provides a closed form expression. 
\begin{proof}[Proof of Lemma \ref{existenceofmaximizer}]
 By Theorem 15.11 in \cite{aliprantisinfinite}, $\mathcal{P}(\Delta)$
is compact in the weak-star topology since $\Delta$ is compact. Further,
the mapping $\phi\rightarrow\mathbb{E}_{\widetilde{\delta}\sim\phi}\left[g\left(x,\,\widetilde{\delta}\right)\right]$
is continuous with respect to the weak-star topology in $\mathcal{P}(\Delta)$
from Assumption \textbf{\textcolor{black}{A5}}, the mapping $\phi\rightarrow D\left(\phi,\,\phi_{u}\right)$
is lower semi-continuous with respect to the weak-star topology in
$\mathcal{P}(\Delta)$ by invoking Theorem 5.27 in \cite{fonseca2007modern},
and so $\phi\rightarrow\mathbb{E}_{\widetilde{\delta}\sim\phi}\left[g\left(x,\,\widetilde{\delta}\right)\right]-\kappa\,D\left(\phi,\,\phi_{u}\right)$
is upper semi-continuous in $\phi\in\mathcal{P}(\Delta)$ with respect
to the weak-star topology. Therefore, the maximizer of $\mathbb{E}_{\widetilde{\delta}\sim\phi}\left[g\left(x,\,\widetilde{\delta}\right)\right]-\kappa\,D\left(\phi,\,\phi_{u}\right)$
is attained in $\phi\in\mathcal{P}(\Delta)$.

Let $\mathcal{M}_{+}\left(\Delta\right)$ denote the space of non-negative
measures on $\Delta$. We note that the regularized cut generation
Problem (\ref{regulcutgen}) is a constrained calculus of variations
problem:

\[
\begin{array}{rcl}
\max_{\phi\in\mathcal{M}_{+}\left(\Delta\right)} &  & \int_{\Delta}g\left(x,\,\delta\right)\phi(\delta)-\kappa\log\left(\frac{\phi(\delta)}{p_{u}}\right)\phi(\delta)d\delta\\
s.t. &  & \int_{\Delta}\phi(\delta)d\delta=1.
\end{array}
\]
By using Euler's equation in the calculus of variations (see Section
7.5 in \cite{luenberger1997optimization}), we obtain after simplification,

\begin{equation}
\kappa\log\left(\phi(\delta)\right)=g\left(x,\delta\right)+C,\,\forall\delta\in\Delta,\label{Euler-1}
\end{equation}
where $C=\upsilon-\kappa\log(p_{u})-\kappa$ and $\upsilon$ is the
Lagrange multiplier of the constraint $\int_{\Delta}\phi(\delta)d\delta=1$.
From (\ref{Euler-1}) and the constraint $\int_{\Delta}\phi(\delta)d\delta=1$,
we obtain the expression

\begin{equation}
\phi_{\kappa,\,x}(\delta)=\frac{\exp\left(g(x,\delta)/\kappa\right)}{\int_{\Delta}\exp\left(g(x,\delta)/\kappa\right)d\delta},\,\forall\delta\in\Delta.\label{phi_Rx}
\end{equation}
\end{proof}
The following lemma is an intermediate result, where we use the Assumption
\ref{fulldimandconvex} that $\Delta$ is full dimensional and convex.
It is used in the proof of Proposition \ref{maxsupdiff}, which paves
the way for the cut generation result for the adaptive sampling scheme.
Recall that $\Gamma(\cdot)$ is the gamma function.
\begin{lem}
\label{importbound}Suppose Assumption \ref{fulldimandconvex} holds.
For any $\kappa\in(0,1]$ and $x\in\mathcal{X},$ we have

\begin{equation}
\int_{\Delta}\exp\left(\frac{g(x,\delta)}{\kappa}\right)d\delta\geq\exp\left(\frac{G(x)}{\kappa}\right)\exp\left(-L_{g,\Delta}(R_{\Delta}+D_{\Delta})\right)\frac{\pi^{d/2}}{\Gamma(d/2+1)}(\kappa R_{\Delta})^{d}.\label{keybound}
\end{equation}
\end{lem}

\begin{proof}
First, we have 
\[
\begin{array}{rcl}
\int_{\Delta}\exp\left(\frac{g(x,\delta)}{\kappa}\right)d\delta & = & \exp\left(\frac{1}{\kappa}\max_{\delta\in\Delta}g(x,\delta)\right)\int_{\Delta}\exp\left(-\frac{1}{\kappa}(\max_{\delta\in\Delta}g(x,\delta)-g(x,\delta))\right)d\delta\\
 & \geq & \exp\left(\frac{1}{\kappa}\max_{\delta\in\Delta}g(x,\delta)\right)\int_{\Delta}\exp\left(-\frac{L_{g,\Delta}}{\kappa}\left\Vert \delta^{\ast}(x)-\delta\right\Vert \right)d\delta,
\end{array}
\]
where $\delta^{\ast}(x)\in\arg\max_{\delta\in\Delta}g(x,\,\delta)$,
and the last inequality follows since $\max_{\delta\in\Delta}g(x,\,\delta)-g(x,\delta)=g(x,\delta^{\ast}(x))-g(x,\delta)$$\leq L_{g,\Delta}\left\Vert \delta^{\ast}(x)-\delta\right\Vert $
due to Assumption \textbf{\textcolor{black}{A5}}. It is then sufficient
to show

\[
\int_{\Delta}\exp\left(-\frac{L_{g,\Delta}}{\kappa}\left\Vert \delta^{\ast}(x)-\delta\right\Vert \right)d\delta\geq\exp\left(-L_{g,\Delta}(R_{\Delta}+D_{\Delta})\right)\frac{\pi^{d/2}}{\Gamma(d/2+1)}(\kappa R_{\Delta})^{d}.
\]
Let $\delta_{\kappa}:=\kappa\delta_{0}+(1-\kappa)\delta^{\ast}(x)$.
Since $\Delta$ is convex by Assumption \ref{fulldimandconvex}, we
deduce

\begin{equation}
B_{\kappa R_{\Delta}}(\delta_{\kappa})=\left\{ \delta:\,\left\Vert \delta-\kappa\delta_{0}-(1-\kappa)\delta^{\ast}(x)\right\Vert \leq\kappa R_{\Delta}\right\} =(1-\kappa)\delta^{\ast}(x)+\kappa B_{R_{\Delta}}(\delta_{0})\subseteq\Delta,\label{Ballinside}
\end{equation}
which implies that, for any $\delta\in B_{\kappa R_{\Delta}}(\delta_{\kappa}),$
there exists $\delta'\in B_{R_{\Delta}}(\delta_{0})$ such that $\delta=\kappa\delta'+(1-\kappa)\delta^{\ast}(x)$.
Then, for any $\delta\in B_{\kappa R_{\Delta}}(\delta_{\kappa}),$
we have

\begin{equation}
\left\Vert \delta^{\ast}(x)-\delta\right\Vert =\kappa\left\Vert \delta'-\delta^{\ast}(x)\right\Vert \leq\kappa\left(\left\Vert \delta'-\delta_{0}\right\Vert +\left\Vert \delta_{0}-\delta^{\ast}(x)\right\Vert \right)\leq\kappa(R_{\Delta}+D_{\Delta}).\label{distbtw}
\end{equation}
Therefore, 
\[
\begin{array}{rcl}
\int_{\Delta}\exp\left(-\frac{L_{g,\Delta}}{\kappa}\left\Vert \delta^{\ast}(x)-\delta\right\Vert \right)d\delta & \geq & \int_{B_{\kappa R_{\Delta}}(\delta_{\kappa})}\exp\left(-\frac{L_{g,\Delta}}{\kappa}\left\Vert \delta^{\ast}(x)-\delta\right\Vert \right)d\delta\\
 & \geq & \exp\left(-L_{g,\Delta}(R_{\Delta}+D_{\Delta})\right)\int_{B_{\kappa R_{\Delta}}(\delta_{\kappa})}1d\delta\\
 & = & \exp\left(-L_{g,\Delta}(R_{\Delta}+D_{\Delta})\right)\frac{\pi^{d/2}}{\Gamma(d/2+1)}(\kappa R_{\Delta})^{d},
\end{array}
\]
where the first inequality is by (\ref{Ballinside}) and the second
is by (\ref{distbtw}), and the equality follows since $\int_{B_{\kappa R_{\Delta}}(\delta_{\kappa})}1d\delta=\frac{\pi^{d/2}}{\Gamma(d/2+1)}(\kappa R_{\Delta})^{d}$
is the volume of the Euclidean ball with radius $\kappa R_{\Delta}$
in $\mathbb{R}^{d}$. 
\end{proof}
Now we are in a position to establish Proposition \ref{maxsupdiff}. 
\begin{proof}[Proof of Proposition \ref{maxsupdiff}]
 By replacing (\ref{phi_Rx}) in $\mathbb{E}_{\widetilde{\delta}\sim\phi}\left[g\left(x,\,\widetilde{\delta}\right)\right]-\kappa\,D\left(\phi,\,\phi_{u}\right)$,
we obtain after simplification,

\begin{equation}
\max_{\phi\in\mathcal{P}(\Delta)}\left\{ \mathbb{E}_{\widetilde{\delta}\sim\phi}\left[g\left(x,\,\widetilde{\delta}\right)\right]-\kappa\,D\left(\phi,\,\phi_{u}\right)\right\} =\kappa\log\left(\int_{\Delta}\exp\left(g(x,\delta)/\kappa\right)d\delta\right)+\kappa\log(p_{u}).\label{innermaxoverprob}
\end{equation}
Applying (\ref{keybound}) to bound the term $\log\left(\int_{\Delta}\exp\left(g(x,\delta)/\kappa\right)d\delta\right)$
in the right hand side of (\ref{innermaxoverprob}), we obtain 
\[
\begin{array}{rcl}
 &  & \max_{\phi\in\mathcal{\mathcal{P}}(\Delta)}\left\{ \mathbb{E}_{\widetilde{\delta}\sim\phi}\left[g\left(x,\,\widetilde{\delta}\right)\right]-\kappa\,D\left(\phi,\,\phi_{u}\right)\right\} \\
 & \geq & \kappa\log\left(\exp\left(\frac{G(x)}{\kappa}\right)\exp\left(-L_{g,\Delta}(R_{\Delta}+D_{\Delta})\right)\frac{\pi^{d/2}}{\Gamma(d/2+1)}(\kappa R_{\Delta})^{d}\right)+\kappa\log(p_{u})\\
 & = & G(x)+\kappa\log(\kappa)d-\kappa C.
\end{array}
\]
Since $\kappa=\min\left\{ \frac{\epsilon}{2C},\left(\frac{\epsilon}{2d}\right)^{2},1\right\} $,
we have 
\[
\kappa\log(\kappa)d-\kappa C\geq\kappa\log(\kappa)d-\frac{\epsilon}{2}\geq-\sqrt{\kappa}d-\frac{\epsilon}{2}\geq-\epsilon,
\]
where the first inequality holds since $\kappa\leq\frac{\epsilon}{2C}$,
the second holds because $\log(\kappa)\geq-\frac{1}{\sqrt{\kappa}}$,
and the last one follows from $\kappa\leq\left(\frac{\epsilon}{2d}\right)^{2}$.
Therefore, we have $\max_{\phi\in\mathcal{\mathcal{P}}(\Delta)}\left\{ \mathbb{E}_{\widetilde{\delta}\sim\phi}\left[g\left(x,\,\widetilde{\delta}\right)\right]-\kappa\,D\left(\phi,\,\phi_{u}\right)\right\} \geq G(x)-\epsilon$.
Moreover, since $\phi_{\kappa,\,x}$ solves the regularized cut generation
Problem (\ref{regulcutgen}), and since the regularization parameter
$\kappa$ and the Kullback-Liebler divergence $D\left(\phi,\,\phi_{u}\right)$
are non-negative, we arrive at the conclusion.
\end{proof}
The bound for cut generation under the adaptive sampling scheme (Proposition
\ref{cutgenerationinexpectation_adaptive}) is an immediate result
from Proposition \ref{maxsupdiff}.
\begin{proof}[Proof of Proposition \ref{cutgenerationinexpectation_adaptive}]
 Since $\delta^{(1)},\delta^{(2)},\ldots,\delta^{(M)}$ are i.i.d.
samples from probability density $\phi_{\kappa(\epsilon),\,x}$, we
have from Proposition \ref{maxsupdiff}, $\mathbb{E}_{\delta^{(i)}\sim\phi_{\kappa(\epsilon),\,x}}\left[g\left(x,\delta^{(i)}\right)\right]\geq G(x)-\epsilon$,
for $i=1,2,\ldots,M$. Therefore, as long as $M\geq1$, we have $\mathbb{E}_{\phi_{\kappa(\epsilon),\,x}^{M}}\left[\max_{1\leq i\leq M}g(x,\delta^{(i)})\right]\geq G(x)-\epsilon$.
\end{proof}
We now prove our main result Theorem \ref{varistep-algo2} (for the
generally convex case) under the adaptive sampling scheme. We need
to use Proposition \ref{cutgenerationinexpectation_adaptive} to control
the error terms.
\begin{proof}[Proof of Theorem \ref{varistep-algo2}]
 From Proposition \ref{cutgenerationinexpectation_adaptive}, we
have $\mathbb{E}_{\phi_{\kappa(\epsilon_{k}),\,x_{k}}^{M_{k}}}[\varepsilon_{k}]\leq\frac{(L_{f}+L_{g,\mathcal{X}})D_{\mathcal{X}}}{\sqrt{k}}$.
Furthermore, $\frac{N}{2}\leq k\leq N$ for $k\in\mathcal{B}$, and
by independence of samples we have 

\[
\mathbb{E}_{\mathcal{P}}\left[\frac{\sum_{k\in\mathcal{B}}\varepsilon_{k}/\sqrt{k}}{\sum_{k\in\mathcal{B}}1/\sqrt{k}}\right]\leq\frac{2D_{\mathcal{X}}(L_{f}+L_{g,\mathcal{X}})}{\sqrt{N}}.
\]
Subsequently, Theorem \ref{varistepresult} gives $\mathbb{E}_{\mathcal{P}}\left[G(\overline{x}_{N,s})\right]\leq14D_{\mathcal{X}}(L_{f}+L_{g,\mathcal{X}})/\sqrt{N}$.
\end{proof}
The proof of Theorem \ref{strong-algo2} (for the strongly convex
case) under the adaptive sampling scheme is as follows. 
\begin{proof}[Proof of Theorem \ref{strong-algo2}]
 From Proposition \ref{cutgenerationinexpectation_adaptive}, we
have $\mathbb{E}_{\phi_{\kappa(\epsilon_{k}),\,x_{k}}^{M_{k}}}[\varepsilon_{k}]\leq\frac{L}{N}\max\left\{ \mu_{f},\mu_{g}\right\} \max\left\{ \frac{L_{f}^{2}}{\mu_{f}^{2}},\frac{L_{g,\mathcal{X}}^{2}}{\mu_{g}^{2}}\right\} $.
It follows that

\[
\mathbb{E}_{\mathcal{P}}\left[\frac{\sum_{k\in\mathcal{B}}k\,\varepsilon_{k}}{\sum_{k\in\mathcal{B}}k}\right]\leq\frac{L}{N}\max\left\{ \mu_{f},\mu_{g}\right\} \max\left\{ \frac{L_{f}^{2}}{\mu_{f}^{2}},\frac{L_{g,\mathcal{X}}^{2}}{\mu_{g}^{2}}\right\} .
\]
Therefore, applying Theorem \ref{strongconvexresult} gives $\mathbb{E}_{\mathcal{P}}\left[G(\overline{x}_{N,s})\right]\leq9L\max\left\{ \mu_{f},\mu_{g}\right\} \max\left\{ \frac{L_{f}^{2}}{\mu_{f}^{2}},\frac{L_{g,\mathcal{X}}^{2}}{\mu_{g}^{2}}\right\} /N$.
\end{proof}

\section{\label{sec:Numerical-Experiments}Numerical Experiments}

This section applies our methods to a simple test problem adapted
from \cite{Calafiore_Uncertain_2005} to illustrate the theory developed
in this paper. Let $\delta_{i}\in\mathbb{R}^{2}$ for all $i=1,\ldots,4$
denote uncertain parameters such that $\left\Vert \delta_{i}\right\Vert \le1$.
We want to solve the following optimization problem: 
\begin{alignat*}{1}
\min\  & -x_{1}-x_{2}\\
\text{s. t. } & \max_{i=1,\ldots,4}\left\{ \max_{\left\Vert \delta_{i}\right\Vert \le1}\left(a_{i}+0.2\delta_{i}\right)^{\top}x-b_{i}\right\} \le0,\\
 & x_{1},x_{2}\in\left[-2,2\right],
\end{alignat*}
where 
\[
\left[\begin{array}{c}
a_{1}^{\top}\\
a_{2}^{\top}\\
a_{3}^{\top}\\
a_{4}^{\top}
\end{array}\right]=\left[\begin{array}{cc}
-1 & 0\\
0 & -1\\
1 & 0\\
0 & 1
\end{array}\right]\ \ \text{and}\ \ \left[\begin{array}{c}
b_{1}\\
b_{2}\\
b_{3}\\
b_{4}
\end{array}\right]=\left[\begin{array}{c}
0\\
0\\
1\\
1
\end{array}\right].
\]

We compare Algorithm 1 with the fixed constraint sampling and the
adaptive constraint sampling schemes, respectively. The parameters
in policy (\ref{variableparameters}) are inherently conservative.
In this experiment, we adjust the parameters $\gamma_{k}$ and $\eta_{k}$
by multiplying them with scaling parameters $c_{g}$ and $c_{e}$,
respectively. These scaling parameters are chosen by doing pilot runs
(see \cite{lan2012validation}). Under fixed constraint sampling,
we set $M_{k}$ to be constant in all iterations, and we consider
$M_{k}\in\left\{ 10,20,50,100\right\} $. Under adaptive constraint
sampling, we generate the probability distribution $\phi_{\kappa,x}$
by the Metropolis-Hastings (MH) algorithm (see e.g. \cite{robert2004monte}),
where we run the MH algorithm for 200 iterations and then take \emph{one}
sample to solve the cut generation problem. 

Table \ref{tab:Results_rule1} reports the results. As we can see,
even though we only generate one sample in each iteration under the
adaptive sampling scheme, the objective value achieved is $-1.560$,
which is close to the true optimal value $-1.559$. Figure \ref{fig:result_1}(a)
illustrates the convergence of the algorithms and Figure \ref{fig:result_1}(b)
shows the constraint violation under different sampling schemes. In
particular, we note that under the fixed sampling scheme, the constraint
violation decreases as the sample size increases. Note that we scale
the parameters $\gamma_{k}$ and $\eta_{k}$ in policy (\ref{variableparameters})
in the experiment, which may result in the failure of Lemma \ref{nonemptyB}.
We see from Figure \ref{fig:result_2}, with the parameter adjustment,
that $\mathcal{B}\neq\emptyset$ and $|\mathcal{B}|$ is at least
linearly increasing in $N$, so that our theoretical analysis is still
valid in this case (which depends on this property of $\mathcal{B}$). 

We generate the probability distribution $\phi_{\kappa,x}$ by the
MH algorithm, and perform sensitivity analysis on the number of iterations
of MH. We provide the associated objective values $f\left(\overline{x}_{N,s}\right)$
by fixing $N=10^{3}$. We can see from Figure \ref{fig:sen_analysis}
that the adaptive sampling scheme achieves a high-performance solution
(with relative gap smaller than $0.1\%$) when the MH algorithm runs
for 200 iterations.

From these experiments, we observe the inherent trade-off between
the two sampling schemes. Under fixed sampling, although only a fixed
sampling distribution is used along all iterations, we need to generate
batch samples to achieve good performance. In contrast, under adaptive
sampling, we need extra effort to generate samples, but only need
one sample at each iteration. 

\begin{table}
\caption{\label{tab:Results_rule1}Simulation results with $10^{3}$ iterations,
$c_{g}=0.35$, and $c_{e}=0.001$.}

\centering{}%
\begin{tabular}{|c|c|c|c|c|c|c|}
\hline 
\multirow{2}{*}{} & \multirow{2}{*}{Adaptive sampling } & \multicolumn{4}{c|}{Fixed constraint sampling} & \multirow{2}{*}{Optimal value}\tabularnewline
\cline{3-6} 
 &  & $M_{k}=10$ & $M_{k}=20$ & $M_{k}=50$ & $M_{k}=100$ & \tabularnewline
\hline 
\hline 
Objective values & $-1.560$ & $-1.621$ & $-1.595$ & $-1.575$ & $-1.566$ & $-1.559$\tabularnewline
\hline 
Relative gaps & $-0.1\%$ & $-4.0\%$ & $-2.3\%$ & $-1.0\%$ & $-0.5\%$ & -\tabularnewline
\hline 
\end{tabular}
\end{table}
\begin{figure}
\begin{centering}
\includegraphics[viewport=15bp 0bp 425bp 320bp,clip,scale=0.57]{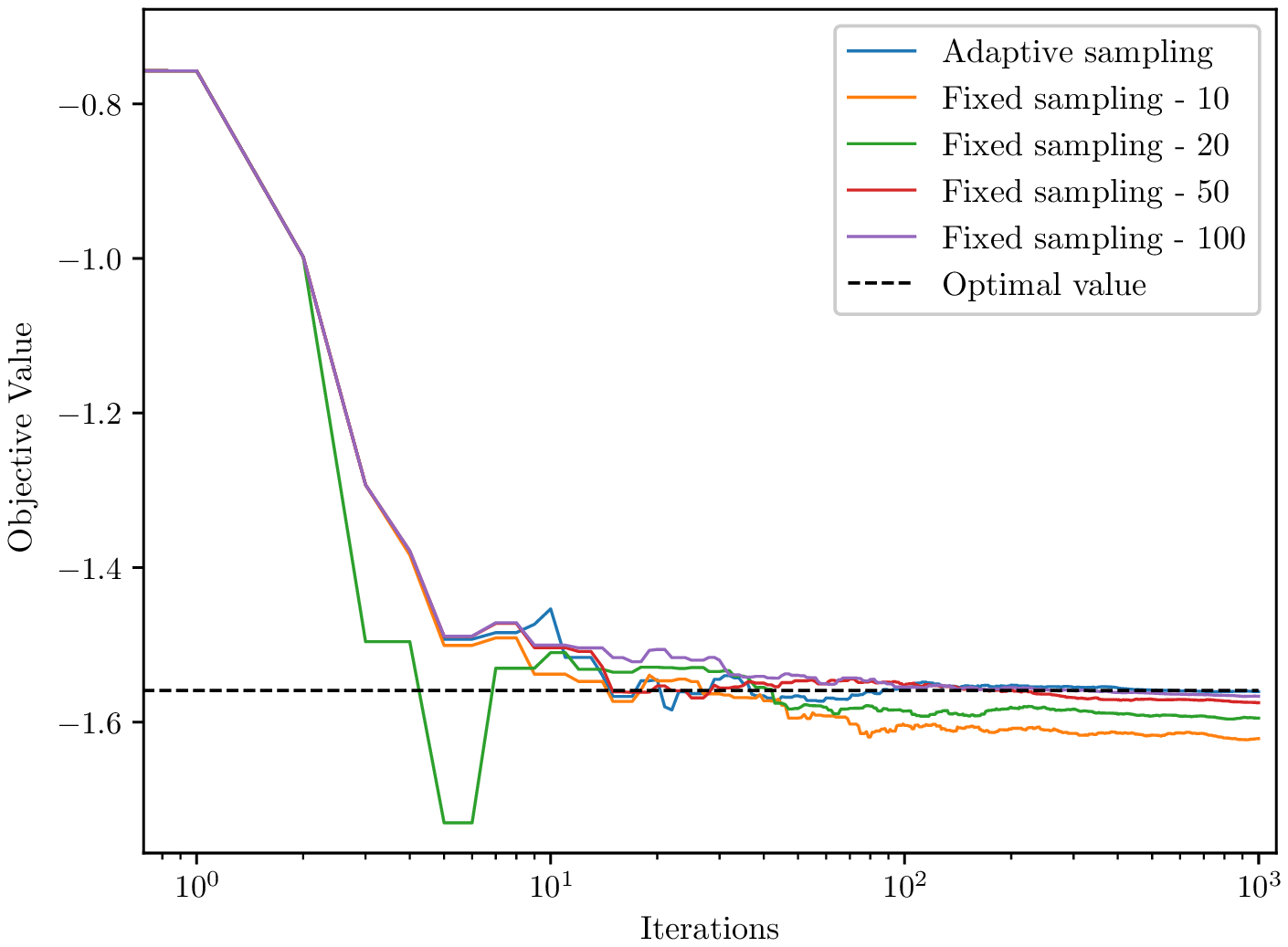}\includegraphics[viewport=15bp 0bp 425bp 320bp,clip,scale=0.57]{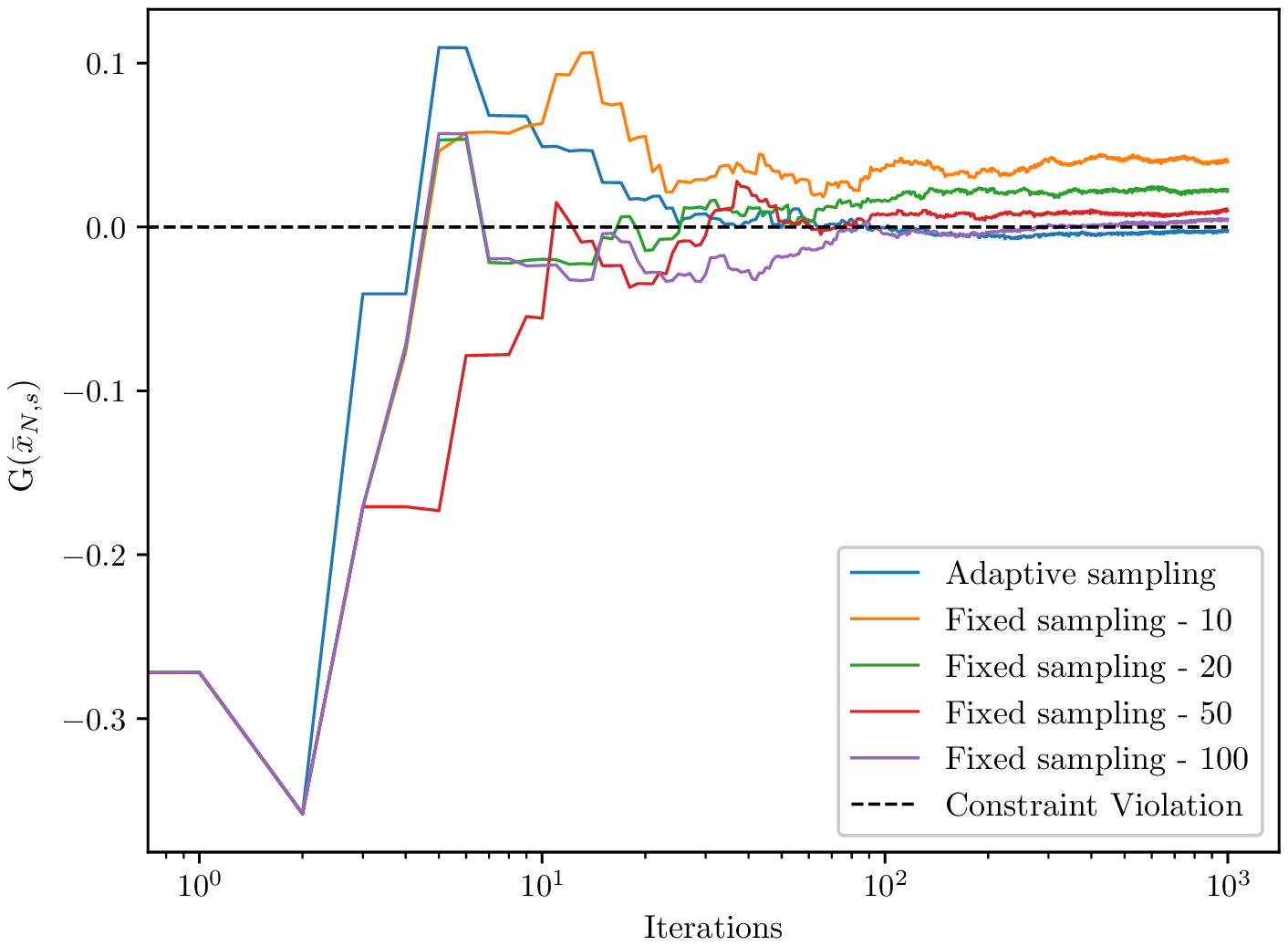}
\par\end{centering}
\caption{\label{fig:result_1}Simulation results given $c_{g}=0.35$ and $c_{e}=0.001$:
(a) objective values and (b) constraint violations with the number
of iterations.}
\end{figure}
\begin{figure}
\begin{centering}
\includegraphics[viewport=15bp 0bp 425bp 320bp,clip,scale=0.57]{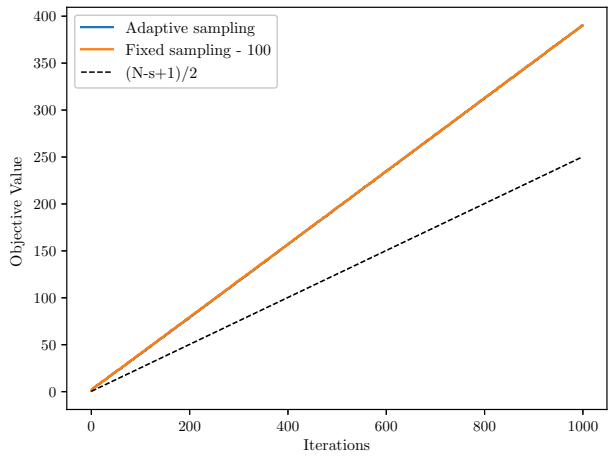}
\par\end{centering}
\caption{\label{fig:result_2}The relation between size $\left|\mathcal{B}\right|$
and the number of iterations.}
\end{figure}
\begin{figure}
\begin{centering}
\includegraphics[viewport=10bp 0bp 425bp 320bp,clip,scale=0.57]{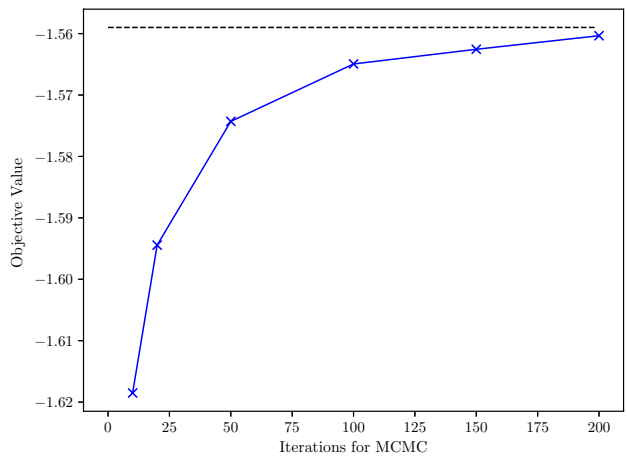}
\par\end{centering}
\caption{\label{fig:sen_analysis}Sensitivity analysis on the iterations of
MH under the adaptive sampling scheme.}
\end{figure}

\section{\label{conclusion}Conclusion}

In this work, we combine CSA (as originally developed in \cite{lan2016algorithms})
with inexact cut generation to solve SIPs. Since the cut generation
problem is typically intractable, we emphasize random constraint sampling
to approximately solve this problem. In our first approach, we rely
on a fixed constraint sampling distribution. Our second approach adaptively
updates the constraint sampling distribution, based on the current
iterate. The major advantage of adaptive over fixed sampling is that,
theoretically, it only requires one sample at each iteration. 

As our main contribution, we provide general error bounds (in terms
of the error in solving each cut generation problem) for inexact CSA.
We show that both our sampling schemes achieve an $\mathcal{O}(1/\sqrt{N})$
rate of convergence in expectation, in terms of both optimality gap
and constraint violation, when the objective and constraint functions
are generally convex. We also improve this rate to $\mathcal{O}(1/N)$
in the strongly convex case.

\bibliographystyle{plain}
\bibliography{References}

\begin{thebibliography}{10}

\bibitem{aliprantisinfinite}
CD~Aliprantis and KC~Border.
\newblock Infinite dimensional analysis. a hitchhiker's guide. 2006.

\bibitem{betro2004accelerated}
Bruno Betr{\`o}.
\newblock An accelerated central cutting plane algorithm for linear
  semi-infinite programming.
\newblock {\em Mathematical Programming}, 101(3):479--495, 2004.

\bibitem{bhat2012non}
Nikhil Bhat, Vivek Farias, and Ciamac~C Moallemi.
\newblock Non-parametric approximate dynamic programming via the kernel method.
\newblock In {\em Advances in Neural Information Processing Systems}, pages
  386--394, 2012.

\bibitem{bonnans2013perturbation}
J~Fr{\'e}d{\'e}ric Bonnans and Alexander Shapiro.
\newblock {\em Perturbation Analysis of Optimization Problems}.
\newblock Springer Science \& Business Media, 2013.

\bibitem{Calafiore_Uncertain_2005}
Giuseppe Calafiore and M.C. Campi.
\newblock Uncertain convex programs: randomized solutions and confidence
  levels.
\newblock {\em Mathematical Programming Series A}, 102:25--46, 2005.

\bibitem{campi2008exact}
Marco~C Campi and Simone Garatti.
\newblock The exact feasibility of randomized solutions of uncertain convex
  programs.
\newblock {\em SIAM Journal on Optimization}, 19(3):1211--1230, 2008.

\bibitem{dentcheva2003optimization}
Darinka Dentcheva and Andrzej Ruszczynski.
\newblock Optimization with stochastic dominance constraints.
\newblock {\em SIAM Journal on Optimization}, 14(2):548--566, 2003.

\bibitem{dentcheva2004optimality}
Darinka Dentcheva and Andrzej Ruszczy{\'n}ski.
\newblock Optimality and duality theory for stochastic optimization problems
  with nonlinear dominance constraints.
\newblock {\em Mathematical Programming}, 99(2):329--350, 2004.

\bibitem{dentcheva2009optimization}
Darinka Dentcheva and Andrzej Ruszczy{\'n}ski.
\newblock Optimization with multivariate stochastic dominance constraints.
\newblock {\em Mathematical Programming}, 117(1):111--127, 2009.

\bibitem{dentcheva2015optimization}
Darinka Dentcheva and Eli Wolfhagen.
\newblock Optimization with multivariate stochastic dominance constraints.
\newblock {\em SIAM Journal on Optimization}, 25(1):564--588, 2015.

\bibitem{esfahani2017infinite}
Peyman~Mohajerin Esfahani, Tobias Sutter, Daniel Kuhn, and John Lygeros.
\newblock From infinite to finite programs: Explicit error bounds with
  applications to approximate dynamic programming.
\newblock {\em arXiv preprint arXiv:1701.06379}, 2017.

\bibitem{esfahani2015performance}
Peyman~Mohajerin Esfahani, Tobias Sutter, and John Lygeros.
\newblock Performance bounds for the scenario approach and an extension to a
  class of non-convex programs.
\newblock {\em IEEE Transactions on Automatic Control}, 60(1):46--58, 2015.

\bibitem{deFarias_Sampling_2004}
Daniela Pucci~de Farias and Benjamin~Van Roy.
\newblock On constraint sampling in the linear programming approach to
  approximate dynamic programming.
\newblock {\em Mathematics of Operations Research}, 29(3):462--478, 08 2004.

\bibitem{fonseca2007modern}
Irene Fonseca and Giovanni Leoni.
\newblock {\em Modern Methods in the Calculus of Variations: L\^{} p Spaces}.
\newblock Springer Science \& Business Media, 2007.

\bibitem{goberna2017recent}
Miguel~A Goberna and MA~L{\'o}pez.
\newblock Recent contributions to linear semi-infinite optimization.
\newblock {\em 4OR}, 15(3):221--264, 2017.

\bibitem{goberna2002linear}
Miguel~A Goberna and Marco~A Lopez.
\newblock Linear semi-infinite programming theory: An updated survey.
\newblock {\em European Journal of Operational Research}, 143(2):390--405,
  2002.

\bibitem{goberna2013semi}
Miguel~{\'A}ngel Goberna and Marco~A L{\'o}pez.
\newblock {\em Semi-infinite programming: Recent advances}, volume~57.
\newblock Springer Science \& Business Media, 2013.

\bibitem{gramlich1995local}
G{\"u}nther Gramlich, Rainer Hettich, and Ekkehard~W Sachs.
\newblock Local convergence of sqp methods in semi-infinite programming.
\newblock {\em SIAM Journal on Optimization}, 5(3):641--658, 1995.

\bibitem{gribik1979central}
PR~Gribik.
\newblock A central-cutting-plane algorithm for semi-infinite programming
  problems.
\newblock In {\em Semi-infinite Programming}, pages 66--82. Springer, 1979.

\bibitem{haskell2013optimization}
William~B Haskell, J~George Shanthikumar, and Z~Max Shen.
\newblock Optimization with a class of multivariate integral stochastic order
  constraints.
\newblock {\em Annals of Operations Research}, 206(1):147--162, 2013.

\bibitem{haskell2017primal}
William~B Haskell, J~George Shanthikumar, and Z~Max Shen.
\newblock Primal-dual algorithms for optimization with stochastic dominance.
\newblock {\em SIAM Journal on Optimization}, 27(1):34--66, 2017.

\bibitem{hettich1993semi}
Rainer Hettich and Kenneth~O Kortanek.
\newblock Semi-infinite programming: theory, methods, and applications.
\newblock {\em SIAM Review}, 35(3):380--429, 1993.

\bibitem{homem2009cutting}
Tito Homem-de Mello and Sanjay Mehrotra.
\newblock A cutting-surface method for uncertain linear programs with
  polyhedral stochastic dominance constraints.
\newblock {\em SIAM Journal on Optimization}, 20(3):1250--1273, 2009.

\bibitem{hu2012sample}
Jian Hu, Tito Homem-de Mello, and Sanjay Mehrotra.
\newblock Sample average approximation of stochastic dominance constrained
  programs.
\newblock {\em Mathematical Programming}, 133(1):171--201, 2012.

\bibitem{ito2000dual}
S~Ito, Y~Liu, and Kok~Lay Teo.
\newblock A dual parametrization method for convex semi-infinite programming.
\newblock {\em Annals of Operations Research}, 98(1):189--213, 2000.

\bibitem{kanamori2012worst}
Takafumi Kanamori and Akiko Takeda.
\newblock Worst-case violation of sampled convex programs for optimization with
  uncertainty.
\newblock {\em Journal of Optimization Theory and Applications},
  152(1):171--197, 2012.

\bibitem{kortanek1993central}
Kenneth~O Kortanek and Hoon No.
\newblock A central cutting plane algorithm for convex semi-infinite
  programming problems.
\newblock {\em SIAM Journal on Optimization}, 3(4):901--918, 1993.

\bibitem{lan2012validation}
Guanghui Lan, Arkadi Nemirovski, and Alexander Shapiro.
\newblock Validation analysis of mirror descent stochastic approximation
  method.
\newblock {\em Mathematical programming}, 134(2):425--458, 2012.

\bibitem{lan2016algorithms}
Guanghui Lan and Zhiqiang Zhou.
\newblock Algorithms for stochastic optimization with expectation constraints.
\newblock {\em arXiv preprint arXiv:1604.03887}, 2016.

\bibitem{li2004smoothing}
Dong-Hui Li, Liqun Qi, Judy Tam, and Soon-Yi Wu.
\newblock A smoothing newton method for semi-infinite programming.
\newblock {\em Journal of Global Optimization}, 30(2-3):169--194, 2004.

\bibitem{lin2017revisiting}
Qihang Lin, Selvaprabu Nadarajah, and Negar Soheili.
\newblock Revisiting approximate linear programming using a saddle point based
  reformulation and root finding solution approach.
\newblock 2017.

\bibitem{ling2010new}
Chen Ling, Qin Ni, Liqun Qi, and Soon-Yi Wu.
\newblock A new smoothing newton-type algorithm for semi-infinite programming.
\newblock {\em Journal of Global Optimization}, 47(1):133--159, 2010.

\bibitem{liu2002adaptive}
Y~Liu and KL~Teo.
\newblock An adaptive dual parametrization algorithm for quadratic
  semi-infinite programming problems.
\newblock {\em Journal of Global Optimization}, 24(2):205--217, 2002.

\bibitem{liu2004new}
Y~Liu, Kok~Lay Teo, and Soon-Yi Wu.
\newblock A new quadratic semi-infinite programming algorithm based on dual
  parametrization.
\newblock {\em Journal of Global Optimization}, 29(4):401--413, 2004.

\bibitem{lopez2007semi}
Marco L{\'o}pez and Georg Still.
\newblock Semi-infinite programming.
\newblock {\em European Journal of Operational Research}, 180(2):491--518,
  2007.

\bibitem{luenberger1997optimization}
David~G Luenberger.
\newblock {\em Optimization by vector space methods}.
\newblock John Wiley \& Sons, 1997.

\bibitem{mehrotra2014cutting}
Sanjay Mehrotra and D{\'a}vid Papp.
\newblock A cutting surface algorithm for semi-infinite convex programming with
  an application to moment robust optimization.
\newblock {\em SIAM Journal on Optimization}, 24(4):1670--1697, 2014.

\bibitem{nemirovski2009robust}
Arkadi Nemirovski, Anatoli Juditsky, Guanghui Lan, and Alexander Shapiro.
\newblock Robust stochastic approximation approach to stochastic programming.
\newblock {\em SIAM Journal on optimization}, 19(4):1574--1609, 2009.

\bibitem{ni2006truncated}
Qin Ni, Chen Ling, Liqun Qi, and Kok~Lay Teo.
\newblock A truncated projected newton-type algorithm for large-scale
  semi-infinite programming.
\newblock {\em SIAM journal on optimization}, 16(4):1137--1154, 2006.

\bibitem{noyan2013optimization}
Nilay Noyan and G{\'a}bor Rudolf.
\newblock Optimization with multivariate conditional value-at-risk constraints.
\newblock {\em Operations Research}, 61(4):990--1013, 2013.

\bibitem{noyan2018optimization}
Nilay Noyan and G{\'a}bor Rudolf.
\newblock Optimization with stochastic preferences based on a general class of
  scalarization functions.
\newblock {\em Operations Research}, 2018.

\bibitem{qi2009smoothing}
Liqun Qi, Chen Ling, Xiaojiao Tong, and Guanglu Zhou.
\newblock A smoothing projected newton-type algorithm for semi-infinite
  programming.
\newblock {\em Computational Optimization and Applications}, 42(1):1--30, 2009.

\bibitem{qi2003semismooth}
Liqun Qi, Soon-Yi Wu, and Guanglu Zhou.
\newblock Semismooth newton methods for solving semi-infinite programming
  problems.
\newblock {\em Journal of Global Optimization}, 27(2-3):215--232, 2003.

\bibitem{reemtsen1991discretization}
Rembert Reemtsen.
\newblock Discretization methods for the solution of semi-infinite programming
  problems.
\newblock {\em Journal of Optimization Theory and Applications}, 71(1):85--103,
  1991.

\bibitem{reemtsen1998numerical}
Rembert Reemtsen and Stephan G{\"o}rner.
\newblock Numerical methods for semi-infinite programming: a survey.
\newblock In {\em Semi-infinite programming}, pages 195--275. Springer, 1998.

\bibitem{robert2004monte}
Christian Robert and George Casella.
\newblock Monte carlo statistical methods springer-verlag.
\newblock {\em New York}, 2004.

\bibitem{shapiro2009semi}
Alexander Shapiro.
\newblock Semi-infinite programming, duality, discretization and optimality
  conditions.
\newblock {\em Optimization}, 58(2):133--161, 2009.

\bibitem{still2001discretization}
Georg Still.
\newblock Discretization in semi-infinite programming: the rate of convergence.
\newblock {\em Mathematical programming}, 91(1):53--69, 2001.

\bibitem{zhang2010new}
Liping Zhang, Soon-Yi Wu, and Marco~A L{\'o}pez.
\newblock A new exchange method for convex semi-infinite programming.
\newblock {\em SIAM Journal on optimization}, 20(6):2959--2977, 2010.

\end{thebibliography}

\end{document}